%% file: arxiv_decentralizedDIANA.tex
\documentclass[10.5pt,letterpaper]{article}
\usepackage{fullpage}

\usepackage[utf8]{inputenc} %
\usepackage[T1]{fontenc}    %
\usepackage{graphicx}
\usepackage{hyperref}       %
\usepackage{url}            %
\usepackage{booktabs}       %
\usepackage{amsfonts}       %
\usepackage{nicefrac}       %
\usepackage{microtype}      %

\usepackage{amssymb}%
\usepackage{pifont}%
\newcommand{\cmark}{\ding{51}}%
\usepackage[round]{natbib}

\usepackage{xcolor}
\definecolor{mydarkblue}{rgb}{0,0.08,0.45}
\hypersetup{ %
    colorlinks=true,
    linkcolor=mydarkblue,
    citecolor=mydarkblue,
    filecolor=mydarkblue,
    }

\let\cite\citep %

\graphicspath{{img/}}

\input{macros.tex}

\commenter{seb}{red}
\commenter{peter}{orange}
\commenter{dmitry}{green}
\commenter{ak}{magenta}
\commenter{mj}{red}

\title{A Linearly Convergent Algorithm for Decentralized Optimization: Sending Less Bits for Free!}

\newcommand{\myenlarge}[2]{\parbox{\widthof{#1}+#2}{\centering #1}} %

\author{%
\myenlarge{Dmitry Kovalev}{1cm} \\ KAUST \and
\myenlarge{Anastasia Koloskova}{1cm} \\ EPFL \and
\myenlarge{Martin Jaggi}{1cm} \\ EPFL \and
\myenlarge{Peter Richt\'{a}rik}{1cm} \\ KAUST \and
\myenlarge{Sebastian U. Stich}{1cm}\\ EPFL
}

\date{}

\begin{document}

\maketitle

\input{main.tex}

\end{document}

%% file: macros.tex
\usepackage{amsfonts}
\usepackage{amsmath}
\usepackage{amsthm}
\usepackage{amssymb}
\usepackage{dsfont}
\usepackage{nicefrac}

\usepackage{xcolor}
\usepackage{color}
\usepackage{graphicx}
\usepackage{verbatim}
\usepackage{xspace} %

\usepackage{algorithm}
\usepackage{algpseudocode}
\usepackage{enumitem}
\usepackage{array}
\usepackage{multirow}
\usepackage{calc}

\usepackage{threeparttable}

\providecommand{\lin}[1]{\ensuremath{\left\langle #1 \right\rangle}}
\providecommand{\abs}[1]{\left\lvert#1\right\rvert}
  \providecommand{\R}{\mathbb{R}} %
  \newcommand{\sym}{\mathbb{S}}

\newcommand{\ones}{\bf 1}

\newcommand{\norm}[1]{\|#1\|}
\newcommand{\dotprod}[2]{\langle #1,#2 \rangle}
\newcommand{\Dotprod}[2]{\left\langle #1,#2 \right\rangle}
\newcommand{\trace}[1]{{\bf tr}\left(#1\right) }

\newcommand{\myspan}[1]{{\operatorname{span}}\left(#1\right) }

\newcommand{\E}[1]{\mathbb{E}\left[#1\right]}
\newcommand{\ED}[2]{\mathbb{E}_{#1}\left[#2\right]}
\newcommand{\EE}{\mathbb{E}}

\newcommand{\cN}{{\cal N}}
\newcommand{\cD}{{\cal D}}  

  \DeclareMathOperator*{\argmin}{arg\,min}

  \providecommand{\0}{\mathbf{0}}
  \providecommand{\1}{\mathbf{1}}

  \providecommand{\mB}{\mathbf{B}}

  \providecommand{\mH}{\mathbf{H}}
  \providecommand{\mI}{\mathbf{I}}

  \providecommand{\mW}{\mathbf{W}}
  \providecommand{\mX}{\mathbf{X}}
  
  \providecommand{\mZ}{\mathbf{Z}}

  \providecommand{\cD}{\mathcal{D}}

  \providecommand{\cM}{\mathcal{M}}
  \providecommand{\cN}{\mathcal{N}}
  \providecommand{\cO}{\mathcal{O}}
  
  \providecommand{\cQ}{\mathcal{Q}}

  \usepackage[textwidth=3.7cm]{todonotes}
  
\providecommand{\mycomment}[3]{\todo[caption={},size=footnotesize,color=#1!20]{\textbf{#2: }#3}}%
\providecommand{\inlinecomment}[3]{%
  {\color{#1}#2: #3}}%
\newcommand\commenter[2]%
{%
  \expandafter\newcommand\csname i#1\endcsname[1]{\inlinecomment{#2}{#1}{##1}}
  \expandafter\newcommand\csname #1\endcsname[1]{\mycomment{#2}{#1}{##1}}
}

\newtheorem{lemma}{Lemma}
\newtheorem{corollary}[lemma]{Corollary}

\newtheorem{remark}[lemma]{Remark}
\newtheorem{assumption}{Assumption}
\newtheorem{theorem}[lemma]{Theorem}
\newtheorem{example}[lemma]{Example}

\usepackage[capitalize,noabbrev]{cleveref}

\usepackage{url}

%% file: main.tex
\begin{abstract}
Decentralized optimization methods enable on-device training of machine learning models without a central coordinator.
In many scenarios communication between devices is energy demanding and time consuming and forms the bottleneck of the entire system. \\
We propose a new randomized first-order method which tackles the communication bottleneck by applying randomized compression operators to the communicated messages. 
By combining our scheme with a new variance reduction technique that progressively throughout the iterations  reduces the adverse effect of the injected quantization noise, we obtain the first scheme that converges linearly on strongly convex decentralized problems while using compressed communication only.
We prove that our method can solve the problems without any increase in the number of communications compared to the baseline which does not perform any communication compression while still allowing for a significant compression factor which depends on the conditioning of the problem and the topology of the network. 
Our key theoretical findings are supported by numerical experiments. %
\end{abstract}

\section{Introduction}
We consider large-scale convex optimization problems of the form
\begin{align}
 f^\star := \min_{x \in \R^d} \frac{1}{n}\sum_{i=1}^n \bigl[ f_i(x) := \ED{\xi\sim \cD}{f_i(x,\xi)} \bigr] \,, \label{eq:1}
\end{align}
with private loss functions $f_i \colon \R^d \to \R$ split among $n$ machines (workers). 
This problem formulation covers for instance empirical risk minimization over finite datasets with equal loss functions but different data samples available on each device, but more generally also the stochastic setting where the workers have access to unbounded number of independent samples.

We assume that the workers are connected over an arbitrary network and that they can only exchange information with their immediate neighbors in the network.
This setting covers the classical parameter-server infrastructures, where all devices are connected to one central server~\cite{dean2012large}, the emerging federated learning paradigm~\cite{McMahan16:FedLearning,McMahan:2017fedAvg,kairouz2019federated}, and most generally, arbitrary decentralized communication topologies~\cite{Tsitsiklis1985:gossip,nedic2020review,Xin20:review}.

Communication is a key bottleneck when the working devices are connected over networks~\cite{seide2016cntk,Alistarh2017:qsgd}.
Quantization techniques enable optimization with compressed messages, hereby reducing the number of bits that have to be exchanges between the workers in each communication round. 
Whilst the first schemes of this type have been presented for centralized topologies only~\cite{Alistarh2017:qsgd,Wangni2018:sparsification}, many adaptations have been developed recently for optimization over arbitrary networks~\cite{Tang2018:decentralized,KoloskovaSJ19gossip,%
Tang2019:squeeze,KoloskovaLSJ19decentralized,Reisizadeh19:exact}.

All these decentralized schemes only converge \emph{sublinearly} when using compressed messages, i.e.\ they need $\cO\bigl(\nicefrac{1}{\epsilon^{\tau}}\bigr)$ iterations to reach accuracy $\epsilon$ for a parameter $0<\tau < \infty$ (most commonly $\tau \in \{1/2,1,2\}$). This is in sharp contrast to centralized approaches with parameter servers, where linear convergence rates of the form $\cO\bigl(\log \nicefrac{1}{\epsilon} \bigr)$ can be attained even with communication compression, for instance when the objective function is strongly convex~\cite{Horvath2019:vr}. 
We believe that there is an intrinsic reason for this limitation: so far, the schemes supporting communication compression and optimization over arbitrary networks have been derived by adapting the the decentralized gradient method and compressing the gradient updates. However, decentralized gradient descent cannot achieve linear convergence on strongly convex problems, even without communication compression \cite{shi2015extra,yuan16:dgd,Koloskova20}.

In this paper we develop a new algorithm for quantized decentralized optimization based on the primal-dual gradient method \cite{chen97:fwsplitting,boyd2011:admm} instead. This new technique allows to overcome limitations of prior schemes. Most importantly, we are able to prove linear convergence on strongly convex functions
for arbitrary unbiased randomized compressors.
Our main contributions can be summarized as follows:
\begin{enumerate}[leftmargin=\widthof{(a)k},nosep]
\item[(a)] We design novel\footnote{After preparation of this manuscript we became aware of parallel work~\cite{Liu2020:pd}. Their proposed algorithm is identical to \textbf{option B} (incremental primal update) in Algorithm~\ref{alg:1}.}
 decentralized optimization algorithms for problem~\eqref{eq:1}.
For $\mu$-strongly convex and $L$-smooth objective with condition number $\kappa := \nicefrac{L}{\mu}$, our main algorithm converges linearly and achieves and $\epsilon$ accurate solution after at most
\begin{align}
 \cO \left( \left(\omega + \kappa(\rho + \omega \rho_\infty) \right) \log\frac{1}{\epsilon}\right) \label{eq:243}
\end{align}
iterations, where  $\rho \geq 1$ denotes the
ratio between the largest and smallest non-zero eigenvalues of the Laplacian gossip matrix that encodes the communication topology, $\rho_\infty \leq \rho$ a new graph parameter we introduce later, and 
 $\omega \geq 0$ quantifies the quality of an arbitrary unbiased quantization operator. For the special case $\omega=0$ (no quantization) 
our rates recover the linear convergence rates of the primal-dual gradient method \cite{Bertsekas1982:book,Alghunaim2019:pd}. We provide further in-depth discussion of our convergence results in Section~\ref{sec:analysis}, see also Tables~\ref{tab:comparison}--\ref{tab:results}.
\item[(b)] Most notably, equation~\eqref{eq:243} reveals that for \emph{any} compression parameter $\omega \leq \min \bigl\{ \rho \rho_\infty^{-1} , \kappa \rho \bigr\}$
the complexity bound is $\cO\bigl(\kappa \rho \log \nicefrac{1}{\epsilon} \bigr)$---the same as for the primal-dual method without compression. This means, that any communication saving achieved by quantization is \emph{for free}, as they do not affect the total number of communication rounds but reduce the number of bits sent every round. We will show that the savings in communication can reach up to a factor of $\cO(n)$ on certain problems.
\item[(c)] We give algorithms and convergence analysis for four important cases: (A) a primal-dual method for dual-friendly problems, (B) an incremental method only using primal gradient oracles, and especially for the machine learning context (C) a method for stochastic gradient oracles and (D) a variance-reduced method when the local functions have finite-sum structure.
\item[(d)] We illustrate in numerical experiments that the performance of our schemes matches with the theoretical rates and compare against prior baselines.
\end{enumerate}

\section{Related Work}

As decentralized optimization problems are special cases of linearly constrained (consensus constraint) optimization problems, algorithms based on  augmented Lagrangian reformulations and primal dual algorithms, such as alternating method of multiplies (ADMM)~\cite{Glowinski1975,Gabay1976}, have been developed early on~\cite{boyd2011:admm}. Linear convergence rates for primal-dual methods on strongly convex problems have been derived and refined over the past decades~\cite{Bertsekas1982:book,Tsitsiklis1985:gossip,chen97:fwsplitting,%
shi14:admm,Alghunaim2019:pd}.
A variety of decentralized optimization schemes have been introduce and studied in the control and optimization communities~\cite{Duchi2012:distributeddualaveragig,Wei2012:distributedadmm,Iutzeler2013:randomizedadmm,%
Rabbat2015:mirrordescent,cola2018nips,%
Lian2017:decentralizedSGD,Wang2018:cooperativeSGD,Koloskova20}, see also the review articles~\cite{sayed14:networks,Xin20:review,nedic2020review}.
Limitations of the distributed gradient method, such as for instance not attaining linear convergence rates, have been pointed out for instance in~\cite{shi2015extra} and techniques such as EXTRA~\cite{shi2015extra} and gradient tracking~\cite{nedic2017achieving} have been developed to achieve linear convergence on strongly convex problems with primal methods as well.
Optimal decentralized algorithms based on accelerated gossip protocols have been presented in~\cite{Scaman2017:optimal} and \cite{Uribe:2018uk}.

\begin{table*}
\caption{Comparison to decentralized algorithms with communication compression and baseline results without compression. The rates show the most significant terms and indicate how many iterations are needed to reach $\norm{x-x^\star}^2\leq \epsilon$ for all nodes. Here $\tilde{\rho} \approx \rho$,
$\tilde \omega \geq \omega$
and $\tau \geq 1$ is an algorithm and function dependent constant, cf.\ the indicated references for definitions.
} 
\label{tab:comparison}
\begin{center}
\begin{minipage}{\linewidth}
\resizebox{0.95\linewidth}{!}{
\begin{tabular}{lccl} \toprule[1pt]
Algorithm\footnote{Convergence rates for the non-accelerated versions of these schemes.} \& Reference & linear rate  & quantization &  convergence to $\epsilon$-accuracy \\
\midrule
Decentralized Gradient Descent~
(Nedi\'{c}+ \citeyear{Nedic2009:distributedsubgrad}; Koloskova+ \citeyear{Koloskova20})
& & &  $\cO\Big( \frac{\sqrt{\kappa} \tilde{\rho}^2 }{\mu} \qquad\: \hspace{4mm} \cdot \frac{1}{\sqrt{\epsilon}} \Big)$ \\[1pt]
QDGD~(Reisizadeh+ \citeyear{Reisizadeh19:exact})
&  & \cmark &  $\cO\Big(\frac{\kappa^2  \tilde{\rho}^4 L^4 + \tilde{\omega}^2}{\mu^2} \hspace{4mm} \cdot \frac{1}{\epsilon^2}  \Big)$   \\[1pt]
Choco-SGD~(Koloskova+ \citeyear{KoloskovaSJ19gossip})
&  & \cmark & $\cO\Big( \frac{\sqrt{\kappa} \tilde{\rho}^2 (1+\omega) }{\mu} \hspace{4mm}  \cdot \frac{1}{\sqrt{\epsilon}} \Big)$ \\[1pt]
EXTRA~(Shi+ \citeyear{shi2015extra}),
Gradient Tracking~(Qu+ \citeyear{qu2016:tracking}; Pu+ \citeyear{Pu2020:tracking})
& \cmark &  & $\cO \Big(( \kappa^\tau \tilde\rho^2 \phantom{\: \omega \kappa \rho_\infty} ) \cdot  \log \frac{1}{\epsilon}\Big)$ \\[1pt]
Primal Dual Gradient Method~%
(Scaman+ \citeyear{Scaman2017:optimal}; Alghunaim~\citeyear{Alghunaim2019:pd})
& \cmark &  &  $\cO \Big( ( \kappa\rho \phantom{\: + \: \omega \kappa \rho_\infty} ) \cdot \log \frac{1}{\epsilon}\Big)$  \\[1pt]
\textbf{this paper} & \cmark & \cmark & $\cO \Big( (\kappa\rho +  \omega \kappa \rho_\infty) \cdot \log \frac{1}{\epsilon}\Big)$ 
 \\[1pt]  \bottomrule[1pt]
\end{tabular}
}
\end{minipage}
\end{center}
\end{table*}

\paragraph{Quantization.}
Quantization techniques allow for (lossy) compression of the messages that are exchanged between the agents to reduce the number of bits that need to be exchanged in each round. 
Quantization has emerged in recent years as an important tool 
in parallel and distributed machine learning~\cite{Seide2015:1bit,Strom2015:1bit,Alistarh2017:qsgd,Wen2017:terngrad}.
Whilst these early schemes have suffered from increased variance due to the randomized compression schemes, schemes based on error-feedback can compensate these effects and attain faster convergence~\cite{Alistarh2018:topk,Stich2018:sparsifiedSGD,KarimireddyRSJ2019feedback,%
StichK19delays} on centralized network topologies.

Quantization in the context of decentralized optimization has first been studied for the decentralized consensus problem where the agents aim to collaboratively compute the average of private data vectors. The effects of various quantization techniques have been studied in~\cite{Xiao2005:drift,Nedic2008:quantizationeffects,%
Carli2010:quantizedconsensus} and many different techniques have been proposed to address quantization errors, such as decreasing stepsizes or adaptive coding schemes
\cite{Carli2010:codingschemes,Yuan2012:distributedquant,Reisizadeh19:exact}. 
Only recently, a first scheme with linear convergence to the exact solution was presented~\cite{KoloskovaSJ19gossip}. However, this algorithm does not converge linearly on arbitrary strongly convex optimization problems that we consider here.
For more general, non-convex problems, further schemes with communication compression have been proposed by~\citet{Tang2018:decentralized,Tang2019:squeeze,KoloskovaLSJ19decentralized}.

\paragraph{Variance Reduction.}
Variance reduction for finite-sum structured problems 
has been introduced in~\cite{Johnson:2013svrg,Defazio2014:saga} and previously been applied to the closely related saddle-point problems~\cite{palaniappan16:svrg} and specifically also for decentralized consensus optimization~\cite{Mokhtari2016:dsa,Xin2019:GTSVRG}.
Recently, optimal algorithms for decentralized finite-sum optimization have been presented~\cite{Hendrikx2020:decentralizedVR}.
Variance reduction and in combination with communication compression has previously been studied in the context of distributed optimization with a parameter server only~\cite{Horvath2019:vr}. This method relies on efficient (and uncompressed) broadcast communication which we avoid here by supporting a fully decentralized topology.

\section{Setup}
We now specify the problem formulation, assumptions, and define several key concepts that will be used throughout the paper.

\subsection{Regularity Assumptions}
\begin{assumption}\label{assumption:convex-and-smooth}
Each cost function $f_i \colon \R^d \to \R$ is $\mu$-strongly convex and $L$-smooth, for parameters 
$0 < \mu \leq L$ and condition number $\kappa := \nicefrac{L}{\mu}$. That is $\forall~x,y~\in~\R^d, i \in [n]$:
 		\begin{align}
 			f_i(y) &\geq f_i(x) + \dotprod{\nabla f_i(x)}{y-x} + \frac{\mu}{2} \norm{y-x}_2^2\,, \label{eq:22}
 		\end{align}
 		\begin{align}
 		f_i(y) &\leq f_i(x) + \dotprod{\nabla f_i(x)}{y-x} + \frac{L}{2} \norm{y-x}_2^2\, . \label{eq:23}
 		\end{align} 		
\end{assumption}

\noindent Sometimes we will also consider the stochastic setting:
 	\begin{align}\label{eq:19}
	 	f_i(x) &= \ED{\xi\sim \cD}{f_i(x,\xi)}\,, & & \forall i \in [n],
 	\end{align}
where only stochastic gradients $\ED{\xi\sim \cD}{\nabla f_i(x,\xi)} = \nabla f_i(x)$ are available. In this case we do need an additional assumption on the strength of the noise:
	 \begin{assumption}[Bounded Variance and Smoothness]\label{assumption:stochastic}
	  Function $f_i(x,\xi)$ is $L$-smooth in expectation and the stochastic variance at the optimum $x^\star := \argmin f(x)$ is bounded. That is,  for all $i \in [n]$ here exist $\sigma_i^2 \in \R_+$, such that
	 	\begin{align}\label{eq:20}
	 		\ED{\xi \sim \cD}{\norm{\nabla f_i(x^\star,\xi) - \nabla f_i(x^\star)}_2^2} \leq \sigma_i^2\,,
	 	\end{align}
	 	is bounded. We define $\sigma^2 := \frac{1}{n}\sum_{i=1}^n \sigma_i^2$. Further, smoothness implies the inequality
	 	\begin{align}\label{eq:21}
 			\ED{\xi\sim\cD}{\norm{\nabla f_i(x,\xi) - \nabla f_i(y,\xi)}_2^2}        
 			\leq 2LB_{f_i}(x,y)\,,
 		\end{align}
 		$\forall x,y \in \R^d, i \in [n],$ where $B_{f_i}(x,y)$ is a Bregman divergence 
   			$B_{f_i}(x,y):=
 			f_i(x) - f_i(y) - \dotprod{\nabla f_i(y)}{x-y}$.
	 \end{assumption} 
 	
 	\begin{remark}
    For the special case of finite-sum structured problems on each worker, $
 		f_i(x) = \frac{1}{m}\sum_{j=1}^m f_{ij}(x)$, equation~\eqref{eq:21} becomes  	
 	\begin{align}\label{eq:31}
 		\textstyle \frac{1}{m} \sum_{j=1}^m \norm{\nabla f_{ij}(x) - \nabla f_{ij}(y)}_2^2 \leq 2LB_{f_i}(x,y)\,,
 	\end{align}
 	 	$\forall x,y \in \R^d$, $i \in [n]$.
 	\end{remark}

\subsection{Optimization over Networks}
We model the network topology as an undirected graph $G=([n],E)$ where $[n]:=\{1,\dots,n\}$ denotes the index set of the agents and $E \subset [n] \times [n]$ a set of pairs of communicating agents, $(i,j) \in E$ if any only if $(j,i) \in E$ (symmetric). If there exists an edge from agent $i$ to agent $j$ they may exchange information along this edge. Thus, agent $i$ may send or receive messages from all its neighbors $\cN_i = \left\{j \in [n] \mid (i,j) \in E \right\}$. 
We encode the communication links in a weighted Laplacian $\mW \in \sym^{n}_+$: 
	\begin{equation}
		\mW_{ij} = \begin{cases}
		-w_{ij}, & i \neq j, (i,j) \in E\\
		0, &i \neq j, (i,j)\notin E\\
		\sum_{l \in \cN_i} w_{il}, & i=j
		\end{cases}.
	\end{equation}
The mixing matrix is positive semidefinite $\mW \in \sym^{n}_+$, respects the graph structure, $\mW_{ij} \neq 0$ only if $(i,j) \in E$, and $\ker \mW = \myspan{\1}$, where $\1 =(1,\dots,1)^\top$. 
We denote by $\lambda_{\rm min}^+(\mW)$ the smallest non-zero eigenvalue of $\mW$ and by $\lambda_{\rm max}(\mW)$ its largest eigenvalue. We define $\rho := \nicefrac{\lambda_{\rm max}(\mW)}{\lambda_{\rm min}^+(\mW)}$ to 
be the ratio between the largest and the smallest non-zero eigenvalue of $\mW$,
and
 $\rho_\infty := \nicefrac{\max_{(i,j)\in E} w_{ij}}{\lambda_{\rm min}^+(\mW)}$ the maximum normalized edge weight.

\begin{remark} \label{rem:bounds}
It holds $\rho_\infty \leq \rho$ and the gap $\rho \rho_\infty^{-1} \geq 1$ can reach size $\Theta(n)$.
\end{remark}
\begin{proof}
For any Laplacian, we have\footnote{Folklore; this bound can be shown by considering Rayleigh quotients $\Delta = \max_{i \in [n]} e_i^\top \mW e_i \leq \lambda_{\rm max}(\mW)$.}
 $\Delta \leq \lambda_{\rm max}(\mW)$ for maximal weighted degree $\Delta:=\max_{i \in [n]} w_{ii}$. 
As $\max_{(i,j)\in E} w_{ij} \leq \max_{i \in [n]} w_{ii} = \Delta$, it follows $\rho_\infty \leq \rho$. For the second claim, consider a $k$-regular graph, for a parameter $1 \leq k \leq n-1$, and uniform weights, $w_{ij}=1$ for $(i,j) \in E$. Then $\max_{(i,j)\in E} w_{ij} = \frac{\Delta}{k}$, and $\rho \rho_\infty^{-1} \geq k$.
\end{proof}

\begin{remark}
The consensus constraint, $x_i = x_j$ can compactly be written as $ \mW\bigl[x_1, \cdots,x_n\bigr] = \0$ in matrix form if the graph is connected.
This observation can be utilized to derive the standard saddle point reformulations of problem~\eqref{eq:1}, see for instance~\cite{Lan2018:decentralized,Alghunaim19}.
\end{remark}

\subsection{Unbiased Quantization}
We consider unbiased randomized quantizers $\cQ \colon \R^d \to \R^d$ as for instance in~\cite{Alistarh2017:qsgd,Wangni2018:sparsification,Horvath2019:vr} with the following assumption on their variance.

\begin{assumption}[$\omega$-quanitzation]\label{a:quantization}
There exists a parameter $\omega \geq 0$ such that for all $x \in \R^d$,
\begin{align}
 \E{ \cQ(x) } &= x\,, &
 \E {\norm{\cQ(x)-x}^2} &\leq \omega \norm{x}^2. \label{def:omega}
\end{align}
\end{assumption}
This general notion comprises many important examples of quantization operators currently used in applications. Below we just name a few (that we later use in the numerical experiments). However, it is important to note that our proposed method does \emph{not} rely a specific choice of quantization operator but can be used in combination with any arbitrary unbiased quantization scheme that satisfies Assumption~\ref{a:quantization}.

\begin{example}[rand-$k$ and dit-$k$]\label{ex:compr} Example compression operators and coding length, assuming that a single floating-point scalar is encoded with $b$ bits with negligible loss in precision.
\begin{itemize}[leftmargin=12pt,nosep]
 \item[--] no compression ($\omega =0$). Each message has size $db$ for this standard baseline.
 \item[--] {\normalfont rand}-$k$: random $k$-sparsification ($\omega = \frac{d}{k}-1$)~\cite{Suresh2017:randomizedDME,Wangni2018:sparsification,Stich2018:sparsifiedSGD}. $\cQ(x):=\frac{d}{k} \cM(x)$, where $\cM(x)$ randomly selects $k$ coordinates of $x$ and masks the others to zero. The sparse vectors can be encoded with $kb + k \log d$ bits (non-zero coordinates and their indices).
 \item[--] {\normalfont dit}-$k$: random $s$-dithering~($\omega = \min\bigl\{\frac{d}{s^2},\frac{\sqrt{d}}{s}\bigr\}$) ~\cite{Goodall1951:randdithering,Roberts1962:randdithering,Alistarh2017:qsgd}.  Each coordinate of the normalized vector $x/\norm{x}$ is randomly rounded to one of $s$ quanitzation levels, (often $s = 2^{k-1}-1$ for integer $k$, so that the levels can be encoded with $k-1$ bits, plus one bit for the sign),
\begin{align*}
\cQ(x) = \textstyle \operatorname{sign}(x) \cdot \norm{x}_2 \cdot \frac{1}{s}\cdot \left\lfloor s \frac{ \abs{x}}{\norm{x}_2} + \xi \right\rfloor %
\end{align*}
for random variable $\xi \sim_{\rm u.a.r.} [0,1]^{d}$. As a special case for $s=2$ one recovers Terngrad~\cite{Wen2017:terngrad}. A trivial upper bound for the encoding length is $dk + b$, but exploiting sparsity (encoding only non-zero quantized values and their indices) this bound can be improved to 
$ \smash{\tilde\cO(s(s+\sqrt{d})} + b)$~\cite{Alistarh2018:topk}. %
\end{itemize}
\end{example}

\section{Algorithm}

\begin{figure*}[t]
\begin{minipage}{\linewidth}
	\begin{algorithm}[H]
		\caption{Four Decentralized Quantized Optimization Algorithms}
		\label{alg:1}
		\begin{algorithmic}[1]
			\State {\bf Initialization}:
			$w_{ij}=w_{ji}  > 0$ for $(i,j) \in E$,
			$z_1^0, \ldots, z_n^0 \in \R^d$ such that $\sum_{i=1}^nz_i^0 = 0$,\\
			$x_1^0, \ldots, x_n^0 \in \R^d$,
			$h_1^0, \ldots, h_n^0 \in \R^d$,
			$\theta>0$, $\alpha  >0$, $\eta > 0$
			\For{$k=0,1,2,\ldots$}
				\For{$i=1,\ldots,n$}{ in parallel on each node \hfill $\triangledown$ 4 options:}
					\State $\bullet$~$x_i^{k+1} = \nabla f_i^\star(z_i^k)$ \Comment{\textbf{Option A} (dual update)} \label{line:1}
					\State $\bullet$~$x_i^{k+1} = x_i^k - \eta(\nabla f_i(x_i^k) - z_i^k)$  \Comment{\textbf{Option B} (incremental primal update)} \label{line:2}
					\State $\bullet$~Sample random $\xi_i^k\sim \cD$
					\State \phantom{$\bullet$~}$x_i^{k+1} = x_i^k - \eta(\nabla f_i(x_i^k, \xi_i^k) - z_i^k)$ \Comment{\textbf{Option C} (stochastic primal update)} \label{line:3}
					\State $\bullet$~Sample $j_i^k \in \{1,\ldots,m\}$ uniformly at random
					\State \phantom{$\bullet$~}$g_i^k = \nabla f_{ij_i^k}(x_i^k) - \nabla f_{ij_i^k}(w_i^k) + \nabla f_i(w_i^k)$
					\State \phantom{$\bullet$~}$w_i^{k+1} = \begin{cases}
						x_i^k,&\text{ with probability } \frac{1}{m}\\
						w_i^k,&\text{ with probability } 1 - \frac{1}{m}
					\end{cases}$ \label{line:5}
					\State \phantom{$\bullet$~}$x_i^{k+1} = x_i^k - \eta(g_i^k - z_i^k)$  \Comment{\textbf{Option D} (finite-sum structured problems)} \label{line:4}
				
					\For{$j \in \cN_i$} 
						\State $\Delta_{ij}^k = \cQ(x_i^{k+1} - h_i^k) + h_i^k$
						\hfill $\triangleright$ prepare quantized dual updates \label{line:6}
					\EndFor
					
					\State $h_i^{k+1} = h_i^k + \alpha \cQ(x_i^{k+1} - h_i^k)$ \hfill{(communication with neighbors)} \label{line:7}
				\EndFor  
				\For{$i=1,\ldots,n$}{ in parallel on each node \hfill {\hfill $\triangledown$ update dual variables}}
					\State $z_i^{k+1} = z_i^k - \theta \sum\limits_{j \in \cN_i} w_{ij}(\Delta_{ij}^k - \Delta_{ji}^k)$ \label{line:8} \hfill{(communication with neighbors)}
				\EndFor
			\EndFor
		\end{algorithmic}
	\end{algorithm}
\end{minipage}	
\end{figure*}

We give the pseudocode for our proposed schemes in Algorithm~\ref{alg:1} above. We will give convergence rates for four different choices of updating the variables $x_i^k$ (in this notation $i \in [n]$ range over the nodes, and $k \geq 0$ over the iterations).

\textbf{Option~A} is applicable only if the dual functions $f_i^* \colon \R^d \to \R$ of each $f_i$ are known at their gradients can  be evaluated efficiently.\footnote{The convex conjugate of $f_i^* \colon R^d \to \R$ of $f_i$ is defined as $f_i^*(z):=\sup_{x \in \R^d} (\lin{x,z}-f_i(x))$.}  

\textbf{Option~B} maintains dual variables $z_i^k$ that are incrementally  updated instead (accessing primal gradient $\nabla f(x_i^k)$ only). Similarly to the incremental version of the classic primal-dual gradient method, we will have  $z_i^k \to z_i^+ := \nabla f_i(x^\star)$ for $k \to \infty$, which explains the intuition behind the $z_i^k$ variables.

 \textbf{Option~C} is applicable when only stochastic gradient oracles are available. 
 
 \textbf{Option~D} applies bias-corrected gradient updates for finite-sum structured $f_i$'s (analogous to the bias corrected updates in SVRG~\cite{Johnson:2013svrg}). Full batch gradients are re-computed after a random number of epochs~\cite{Hannah2018:breaking}.
 
 We give the convergence rates for these variants  in Section~\ref{sec:analysis} below (see also Table~\ref{tab:results}).

The updates on lines~\ref{line:1}--\ref{line:4} (depending on the chosen option) are performed in parallel on each agent. The auxiliary vectors $h_i^k$ updated on line \ref{line:7} are crucial component in our scheme that are required to achieve linear convergence: we will show in the appendix that $\smash{h_i^k} \to x^\star$ for ($k \to \infty$), so that for the quantization on line~\ref{line:6} we will be able to show (by virtue of~\eqref{def:omega}) that the quantization noise reduces linearly to zero as $x_i^k \to x^\star$ for ($k \to \infty$). This would not be possible when quantizing the iterates $x_i^k$ directly.

\paragraph{Implementation Details.}
It is easy to see that only quantized vectors need to be exchanged between the clients (every node needs to send two quantized vectors to each of its neighbors). To see this, assume that the vectors $h_i^{k}$ are known to all neighbors of node $i$ (maintaining $h_i^k$ requires only quantized updates as per line~\ref{line:7}: $h_i^{k+1}-h_i^k = \alpha q$, where $q$ is a quantized vector). The update on line~\ref{line:8} can be rewritten as $$z_i^{k} - z_i^{k+1} = \theta \sum_{j \in \cN_i} \left(h_i^k - h_j^k + q_i - q_j \right),$$ where $q_i,q_j$ are quantized vectors. Further note that the memory requirement is quite low per node: each agent needs to store its local copies of $x_i^k, z_i^k, h_i^k$ and $\bar h_i^k := \frac{1}{|\cN_i|} \sum_{j \in \cN_i} h_j^k$ (but not each $h_j^k$ individually). This memory efficient implementation is similar as the one explained in~\cite{KoloskovaSJ19gossip}.

\begin{table*}[t]
\resizebox{\linewidth}{!}{
\begin{tabular}{lll}
\toprule
Setting  & Convergence Rate, $\tilde \cO(\cdot)$ hides logarithmic factors & Reference \\ \midrule
\textbf{A}, dual $\nabla f_i^\star(z)$ available %
& \multirow{2}{*}{$\cO \left( (\omega + \kappa(\rho + \omega \rho_\infty))\log \frac{1}{\epsilon} \right)$} & Theorem~\ref{thm:A} \\
\textbf{B}, primal $\nabla f_i(x)$ available %
& &  Theorem~\ref{thm:B} \\
\textbf{C}, stochastic $\nabla f_i(x,\xi)$ available %
& $\tilde \cO \left(\omega + (\rho + \omega \rho_\infty)\left(\kappa + \frac{\sigma \sqrt{1+\omega}}{\sqrt{\epsilon} \mu} + \frac{\sigma^2 (\rho + \omega \rho_\infty)}{\epsilon \mu^2} \right)  \right)$ & Theorem~\ref{thm:C} \\
\textbf{D}, finite sum $f_i(x)=\frac{1}{m}\sum_{j=1}^m f_{ij}(x)$ %
& $\cO \left( (m + \omega + \kappa(\rho + \omega \rho_\infty))\log \frac{1}{\epsilon} \right)$ & Theorem~\ref{thm:D} \\
\bottomrule
\end{tabular}
}
\caption{Summary of the convergence results for Algorithm~\ref{alg:1} to reach accuracy $\norm{x - x^\star}^2 \leq \epsilon$ on all nodes. The depicted results are for stepsizes $\alpha = \frac{1}{\omega +1}$, $\eta=\frac{1}{L}$ and $\theta =\Theta\bigl(\frac{\mu}{\lambda_{\max}(\mW) + \omega\max_{(i,j)\in E}w_{ij}}\bigr)$ for option A, B and D and chosen as in equation~\eqref{eq:26} for option C. }
\label{tab:results}
\end{table*}	

\section{Convergence Analysis}
\label{sec:analysis}

We summarize the convergence results of Algorithm~\ref{alg:1} in Theorem~\ref{thm:results}. All proofs are given in the supplementary materials, restated as Theorems~\ref{thm:A}, \ref{thm:B}, \ref{thm:C}, \ref{thm:D}. %

\begin{theorem}\label{thm:results}
	Under Assumptions~\ref{assumption:convex-and-smooth}--\ref{a:quantization}, for any given $\epsilon > 0$ Algorithm~\ref{alg:1} with stepsizes $\alpha = \frac{1}{\omega +1}$, $\eta=\frac{1}{L}$ and $\theta =\Theta\bigl(\frac{\mu}{\lambda_{\max}(\mW) + \omega\max_{(i,j)\in E}w_{ij}}\bigr)$ for option A, B and D and chosen as in equation~\eqref{eq:26} for option C reaches accuracy $\norm{x - x^\star}^2 \leq \epsilon$ on all nodes after the following number of iterations $T$:
	
	\textbf{Options A/B:} dual $\nabla f_i^\star(z)$ / primal $\nabla f_i(x)$ available
	\begin{align*}
	T= \cO \left( (\omega + \kappa(\rho + \omega \rho_\infty))\log \frac{1}{\epsilon} \right)
	\end{align*}
	
	\textbf{Option C:} stochastic $\nabla f_i(x,\xi)$ available\par \vspace{-0.3cm}
	\resizebox{\linewidth}{!}{\parbox{1.1\linewidth}{
	\begin{align*}
	T=\tilde \cO \left(\omega + (\rho + \omega \rho_\infty)\left(\kappa + \frac{\sigma \sqrt{1+\omega}}{\sqrt{\epsilon} \mu} + \frac{\sigma^2 (\rho + \omega \rho_\infty)}{\epsilon \mu^2} \right)  \right)
	\end{align*}
	}}
	
	\textbf{Option D:} finite sum $f_i(x)=\frac{1}{m}\sum_{j=1}^m f_{ij}(x)$
	\begin{align*}
	T=\cO \left( (m + \omega + \kappa(\rho + \omega \rho_\infty))\log \frac{1}{\epsilon} \right),
	\end{align*}
	where $\tilde \cO(\cdot)$ hides logarithmic factors.
\end{theorem}

For \textbf{Option A} and \textbf{B} we  obtain the same linear convergence rate. 
For $\omega = 0$ the rate simplifies to $\tilde \cO \bigl( \kappa \rho  \bigr)$, which is the product of the condition number of $f$ (difficulty of the optimization problem) and the spectral gap of $\mW$ (how fast information diffuses in the graph). 

The dependence on these parameters can be improved to their square roots with accelerated gradient methods cf.~\cite{Scaman2017:optimal,Uribe:2018uk}. In particular, our scheme fits the Catalyst framework~\cite{lin15:catalyst} that can potentially be used to derive optimal accelerated rates with restarts. However, indirect acceleration via Catalyst might not give the best practical scheme, and direct acceleration would be preferred (though not derived in this work).

Our result recovers the best known rates for non-accelerated algorithms~\cite{Alghunaim2019:pd} and in the centralized setting ($\rho = 1$) we recover the rate of the standard gradient method. 
In contrast to the method proposed in~\cite{Reisizadeh:2018DecentralizedQuantized,Reisizadeh19:exact} we are here able to show linear convergence for our scheme even with quantization, i.e., for $\omega > 0$.
\citet{Liu2020:pd} independently show convergence rate $\tilde \cO\bigl(\kappa \rho \omega \bigr)$ for \textbf{option B}, whereas in our result, $\cO\bigl(\kappa \rho + \kappa \rho_\infty \omega \bigr)$, the dependency on $\omega$ can be weaker.

The linear $\cO(\omega)$ term that appears in all our results is not crucial, as sending $\omega$-quantized vectors is typically $\cO(\omega)$ times faster than sending uncompressed vectors (consider random-$k$ quantization as a guiding example), thus $\cO(\omega)$ is proportional to the time it takes to send one single unquantized vector between two nodes.

\paragraph{Compression for free.} Note that for any choice of $\omega$ for which $\omega + \kappa (\rho + \omega \rho_\infty) = \cO\bigl(\kappa \rho \bigr),$ or in other words, 
\begin{align}
\omega \leq \min \bigl\{\rho \rho_\infty^{-1}, \kappa \rho \bigr\}\,,
 \label{eq:min}
\end{align}
the total number of iteration does not increase but the number of bits send in each iteration can be decreased.

 As explained in Remark~\ref{rem:bounds}, the ratio $\rho \rho_\infty^{-1}$ can reach size $\Theta(n)$, in particular for $k$-regular graphs with uniform weights, $\rho \rho_\infty^{-1} = \Theta(k)$. Hence 
$\omega$ can be chosen as large as $\Theta(n)$ for graphs with large maximal degree 
$\Delta$. As a second example, consider a star graph with a central node connected to all other nodes and uniform edge weights, $\lambda_{\rm max}(\mW)$ and the the spectral gap are both of order $\Omega(n)$, so that the choice $\omega = \cO(n)$ is admissible. For well connected graphs (such as regular graphs), the second term in~\eqref{eq:min} becomes smaller, but for difficult optimization problems with $\kappa=\Omega(\Delta)$ we see that compression up to $\omega=\cO(\Delta)$ is possible without affecting the convergence rate.

In \textbf{option D} we leverage the finite-sum structure of $f_i$. In each iteration only a single new gradient $\nabla f_{ij}$ has to be computed (unless a full pass over the local dataset is triggered). Our method combines SVRG-style variance reduction (reducing the variance of the stochastic gradients) with our new variance reduction technique for quantized communication to achieve linear convergence on decentralized networks. For $\omega =0$ and $\rho=1$ we recover the convergence rate of SVRG and for $\rho>1$ our rate improves over the $\tilde \cO \bigl(m + \kappa^2 \rho^2\bigr)$ convergence rate of the recently proposed GT-SVRG~\cite{Xin2019:GTSVRG} which does not support quantization.

For \textbf{option C}, with stochastic updates, 
we observe that our convergence rate recovers the linear rate of option \textbf{A} and \textbf{B}, when $\sigma^2 \to 0$. However, when $\sigma^2$ is large, the rate is dominated by the $\cO\bigl(\frac{\sigma^2}{\epsilon} \bigr)$ term, and the algorithm only converges sublinearly.

\begin{figure*}[t!]
		\hfill
	\centering
	\includegraphics[height=3.9cm]{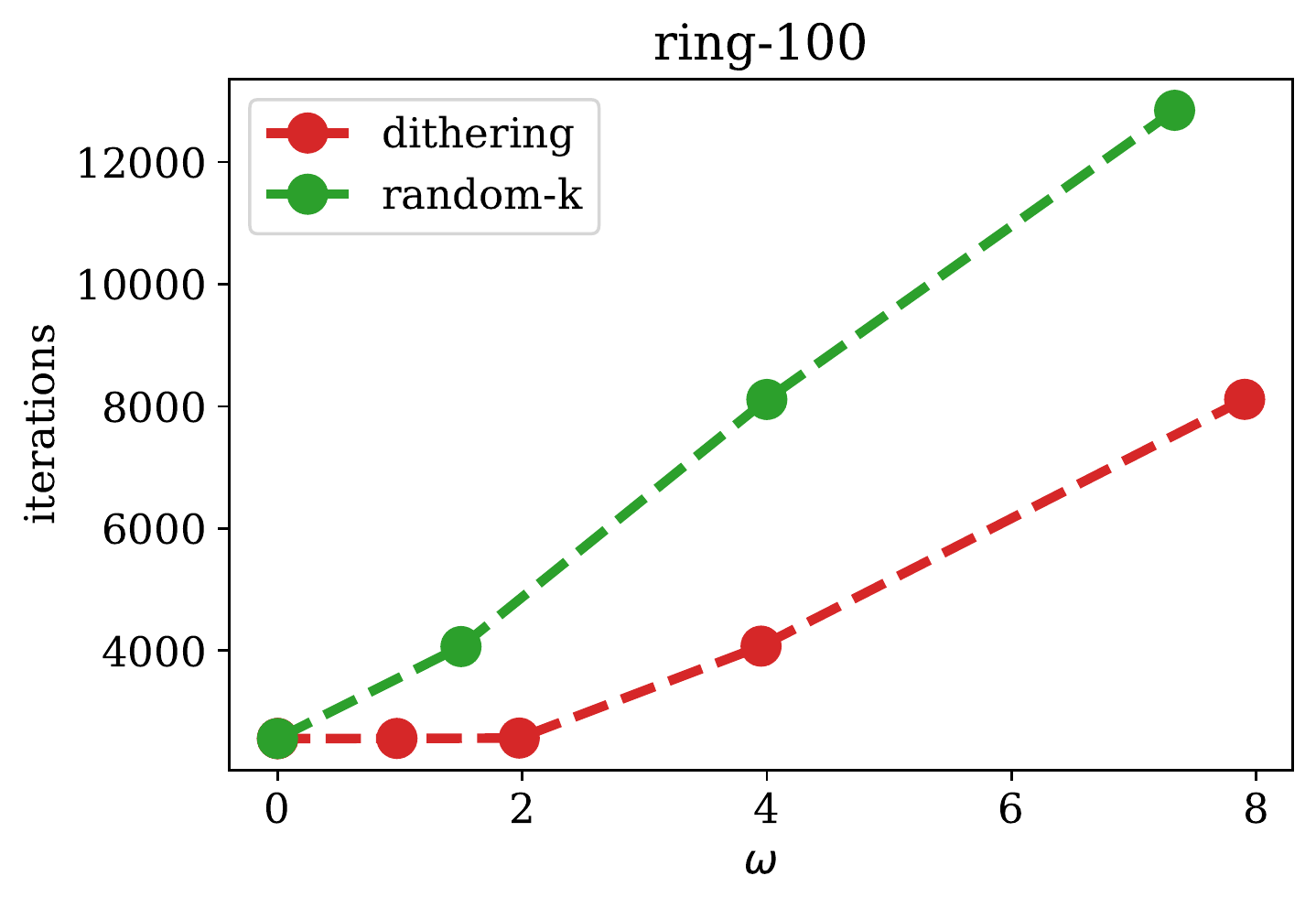}
	\hfill
	\includegraphics[height=3.9cm]{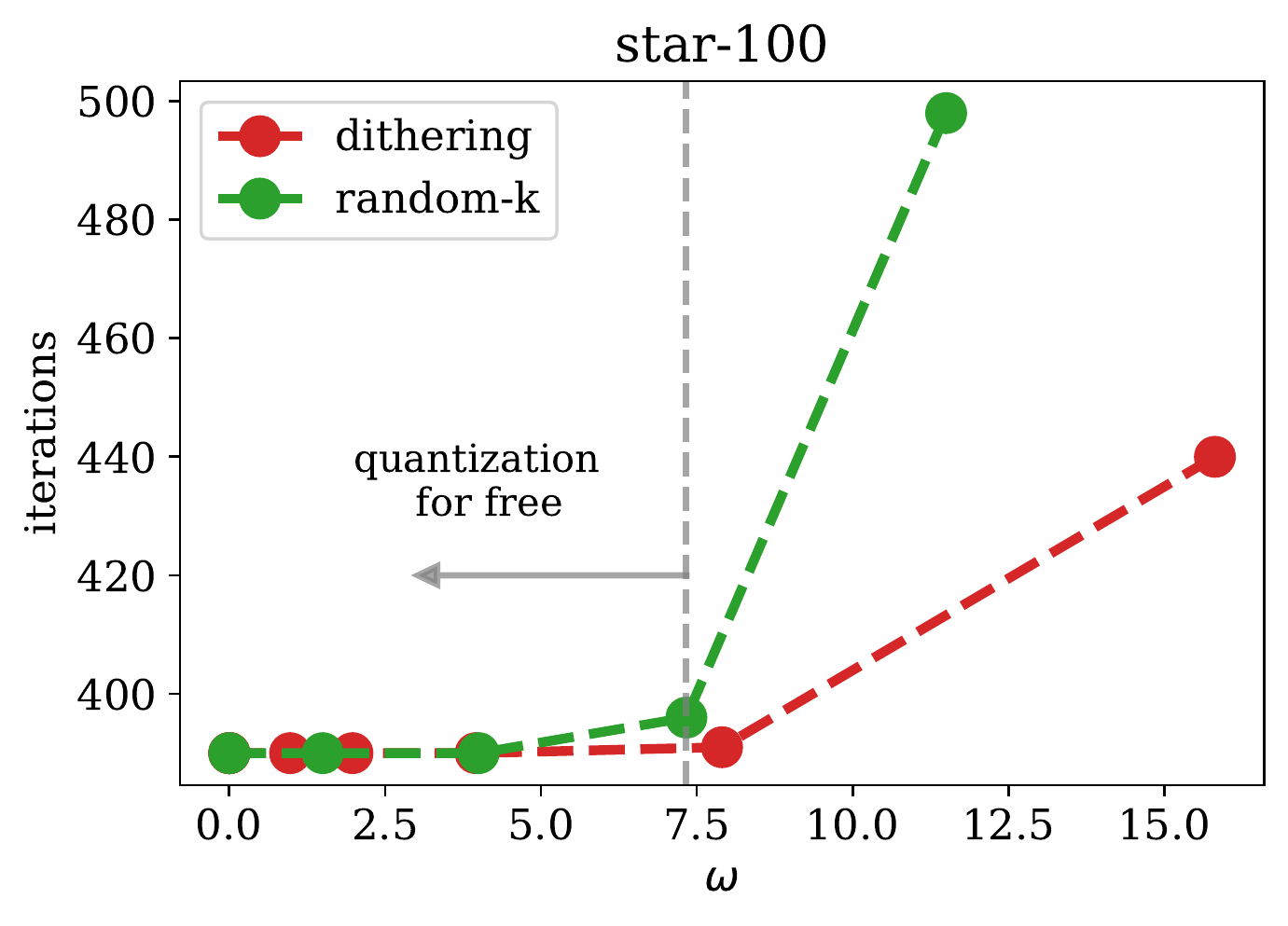}
	\hfill\null
	\vspace{-3mm}
	\caption{Illustrating quantization for free (right vs.\ left).
	 Iterations to converge to $10^{-3}$ error for Algorithm~\ref{alg:1} (option B) with different quantization functions. Average consensus problem on the \emph{star} and \emph{ring} topologies with $n=100$ nodes, $d=250$ and (rand-$k$) and (dit-$k$) compression.	 
	 }\label{fig:free}
\end{figure*}

\begin{figure*}[t!]
	\hfill
	\centering
	\includegraphics[height=3.9cm]{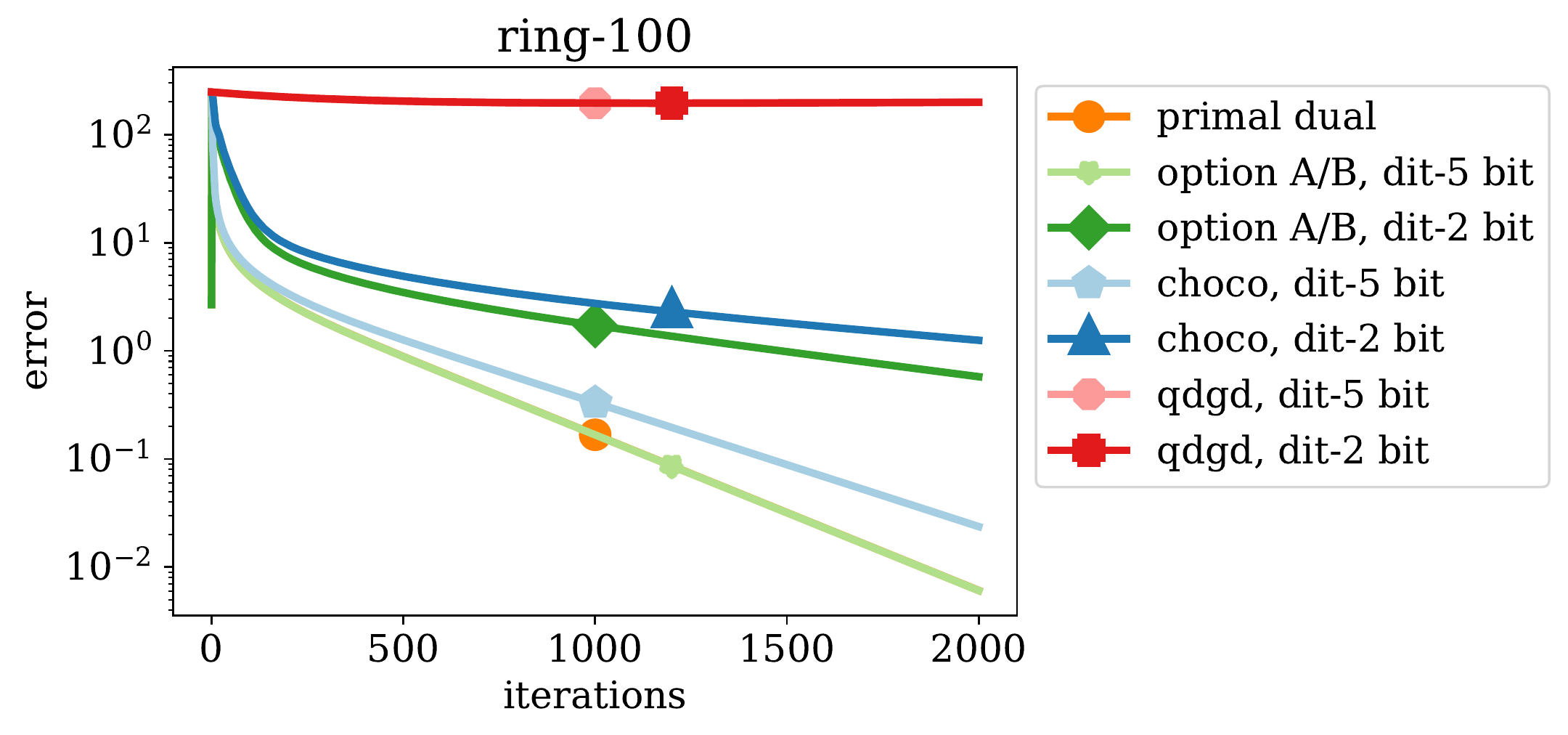}
	\hfill
	\includegraphics[height=3.9cm]{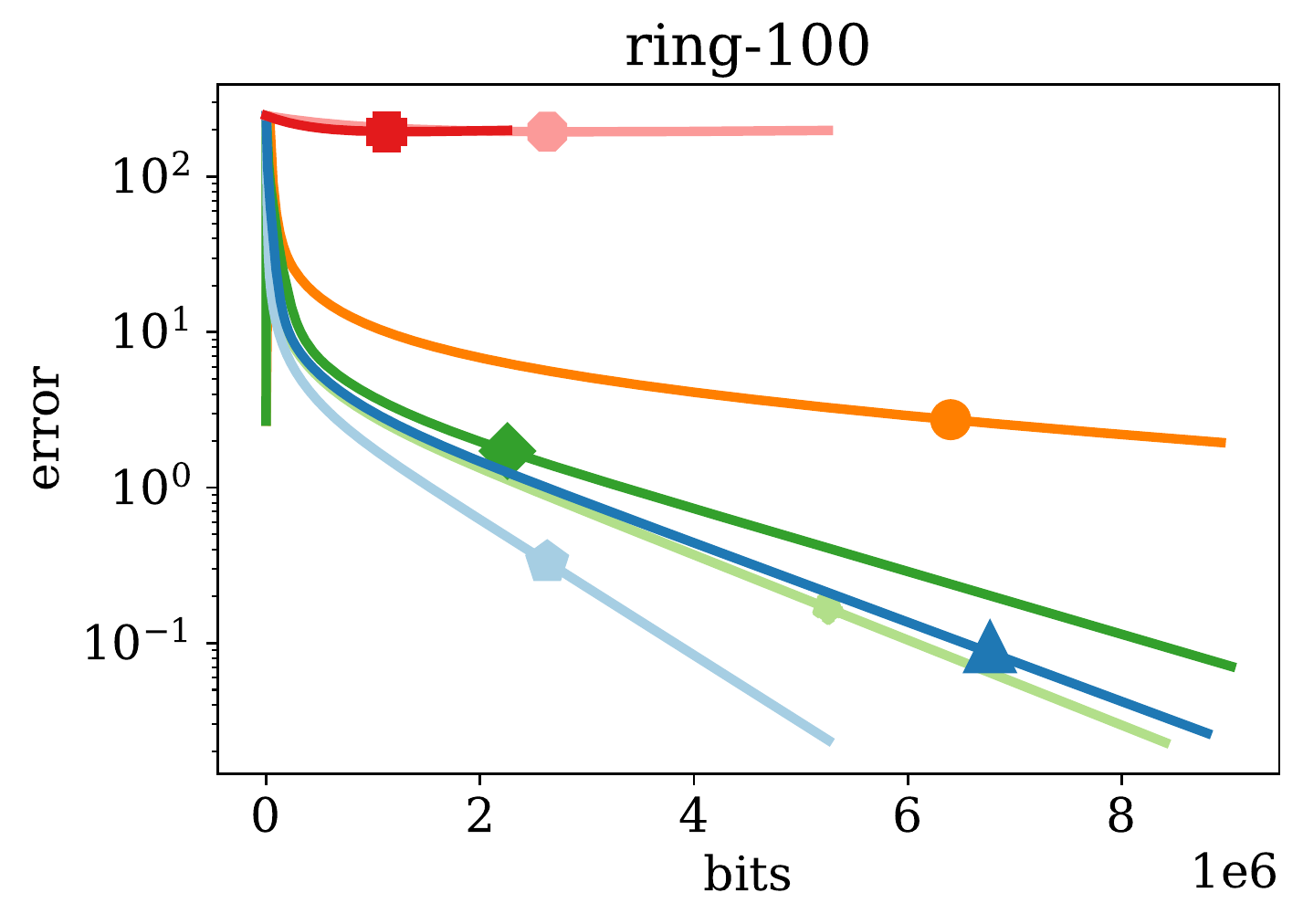}
	\hfill \null
    \\
    \hfill
	\includegraphics[height=3.9cm]{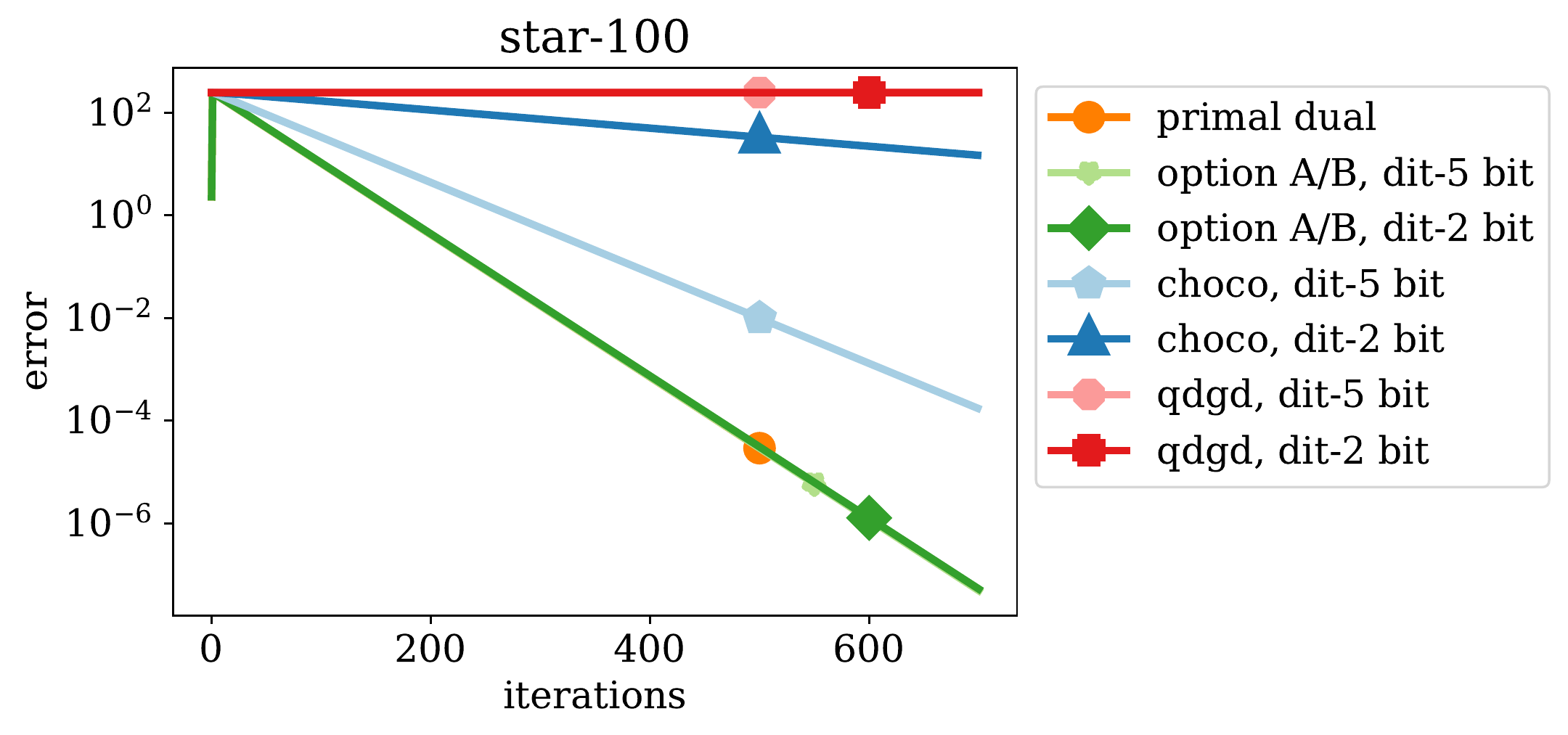}
	\hfill
	\includegraphics[height=3.9cm]{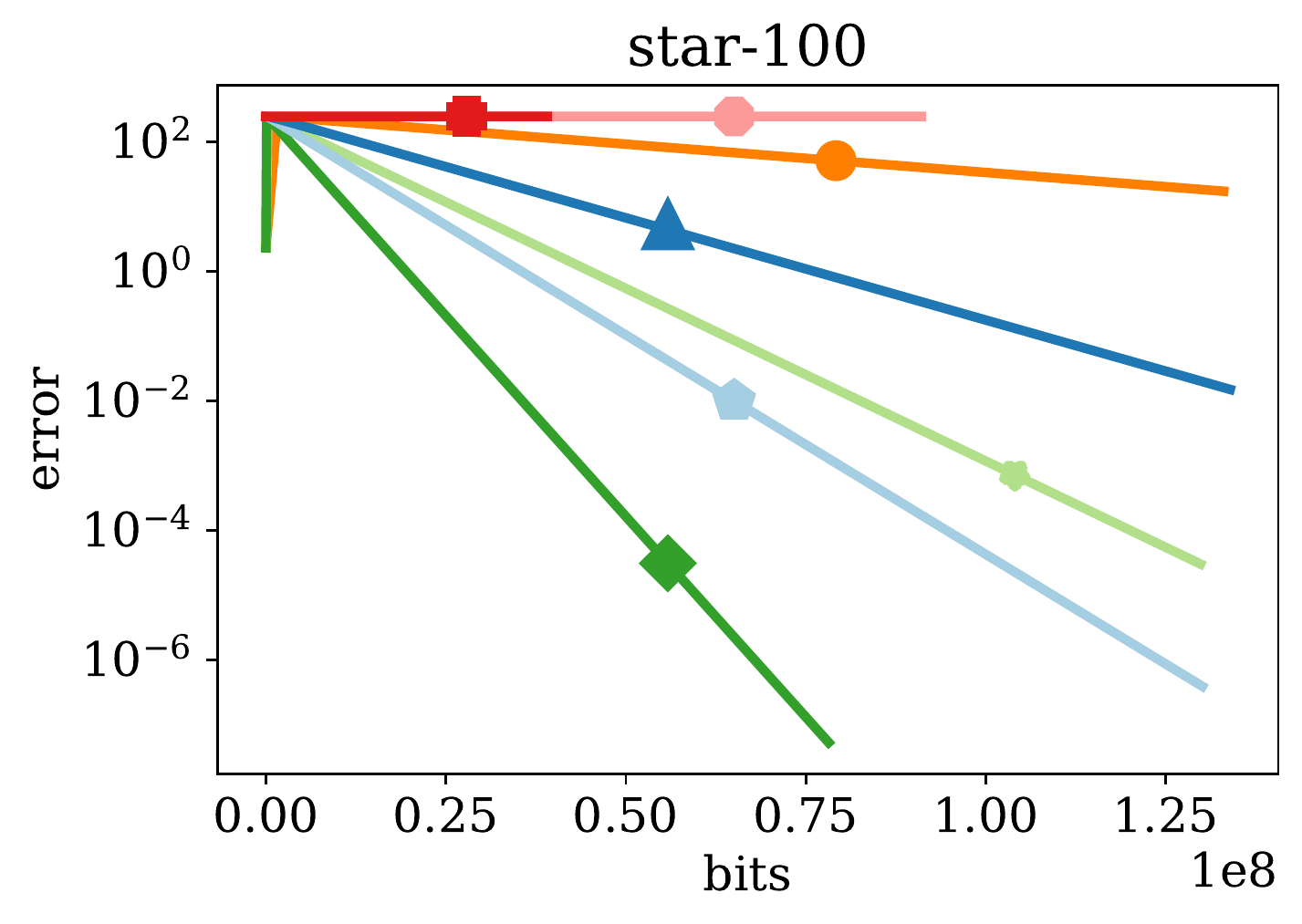}
	\hfill\null
	\vspace{-3mm}
	\caption{Comparison to the baselines. Average consensus problem on the star and ring topologies with $n=100$ nodes, $d=250$ and (rand-$k$) and (dit-$k$) compression.}\label{fig:average_baselines}
\end{figure*}

\section{Experiments}
In this section we experimentally validate our theoretical findings.

\paragraph{Setup.} We use rand-$k$ and dit-$k$ quantization functions (see Example~\ref{ex:compr}). 
We choose two unweighted ($w_{ij} = 1$ for $(i,j) \in E$) graphs on $n$ nodes for our experiments: The \emph{ring}, where every node is connected to two neighbours. As it holds $\rho \approx \rho_\infty \approx n^2$ we see that this is a challenging topology, only allowing communication compression for $\omega=\cO(1)$ (see also Remark~\ref{rem:bounds}). Further, the  \emph{star} graph, where $(n-1)$ nodes have no direct links between them, but are all connected to the central node. Here it holds $\rho=n$, $\rho_\infty=1$ and compression for $\omega = \cO(n)$ is suggested by our theory.

As baselines we use decentralized gradient descent algorithms with quantized communications designed for convex cases: QDGD \cite{Reisizadeh19:exact}, Choco-Gossip and Choco-SGD \cite{KoloskovaSJ19gossip} for consensus and logistic regression correspondingly. Note that when the compression function is identity ($\omega = 0$), Choco-SGD recovers D-SGD \cite{Nedic2009:distributedsubgrad}, and our Algorithm~\ref{alg:1} recovers Primal Dual GD \cite{Scaman2017:optimal,Alghunaim2019:pd}. In all our experiments we tune the hyperparameters of these algorithms independently over a logarithmic grid. 

\paragraph{Average consensus.}

\begin{figure*}[t]
	\hfill
	\centering
	\includegraphics[height=3.9cm]{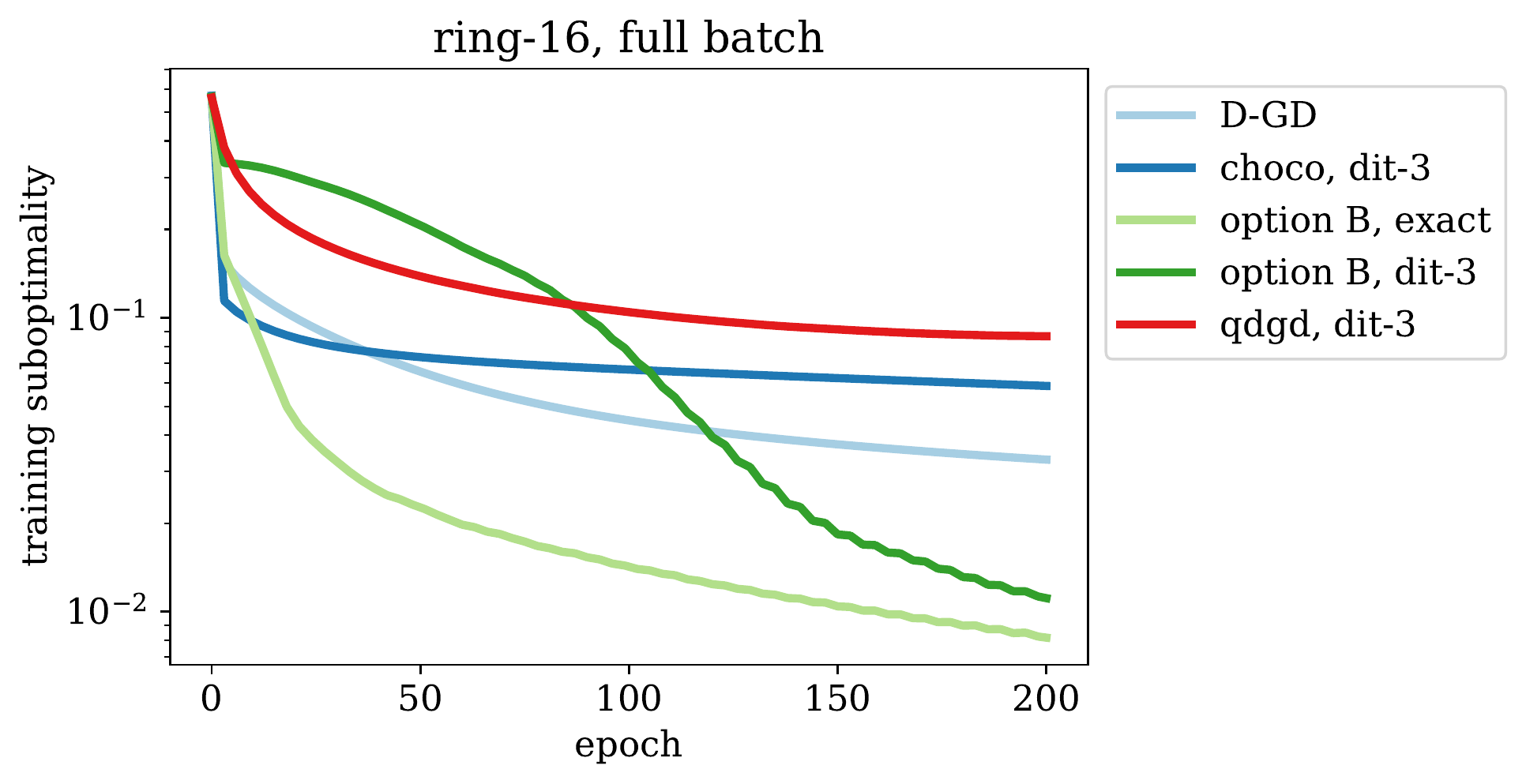}
	\hfill
	\includegraphics[height=3.9cm]{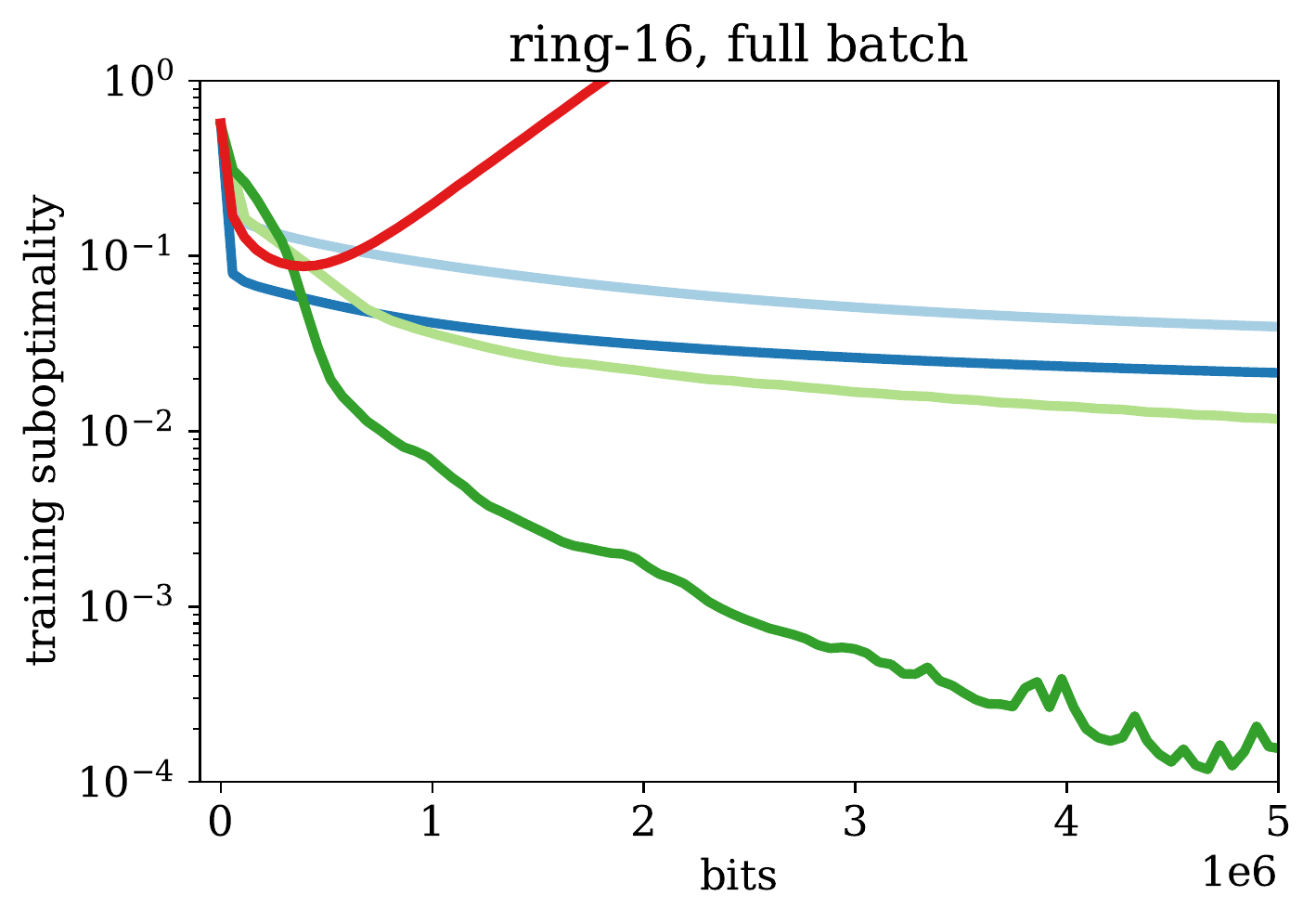}
	\hfill\null
	\\
	\hfill
	\centering
	\includegraphics[height=3.9cm]{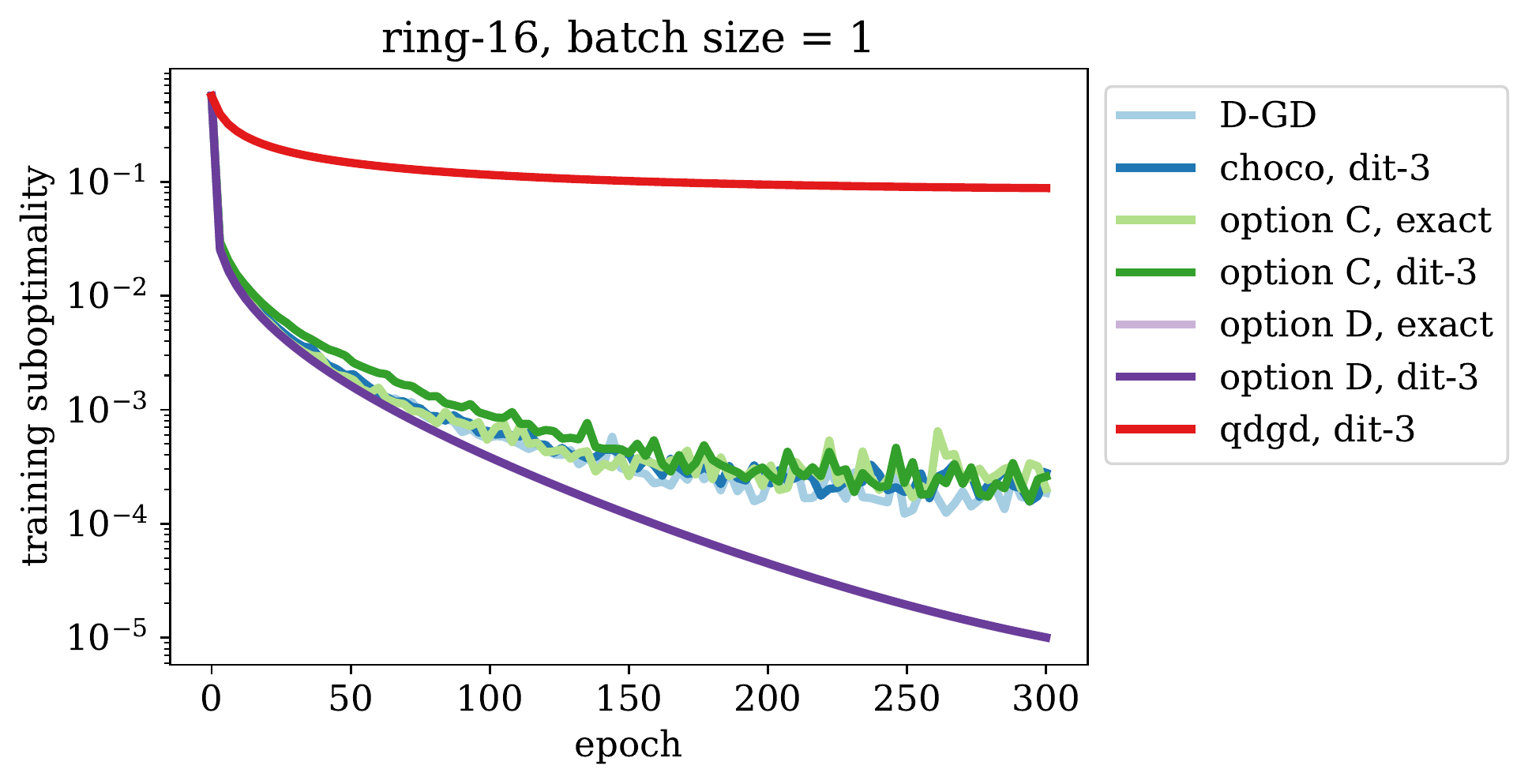}
	\hfill
	\includegraphics[height=3.9cm]{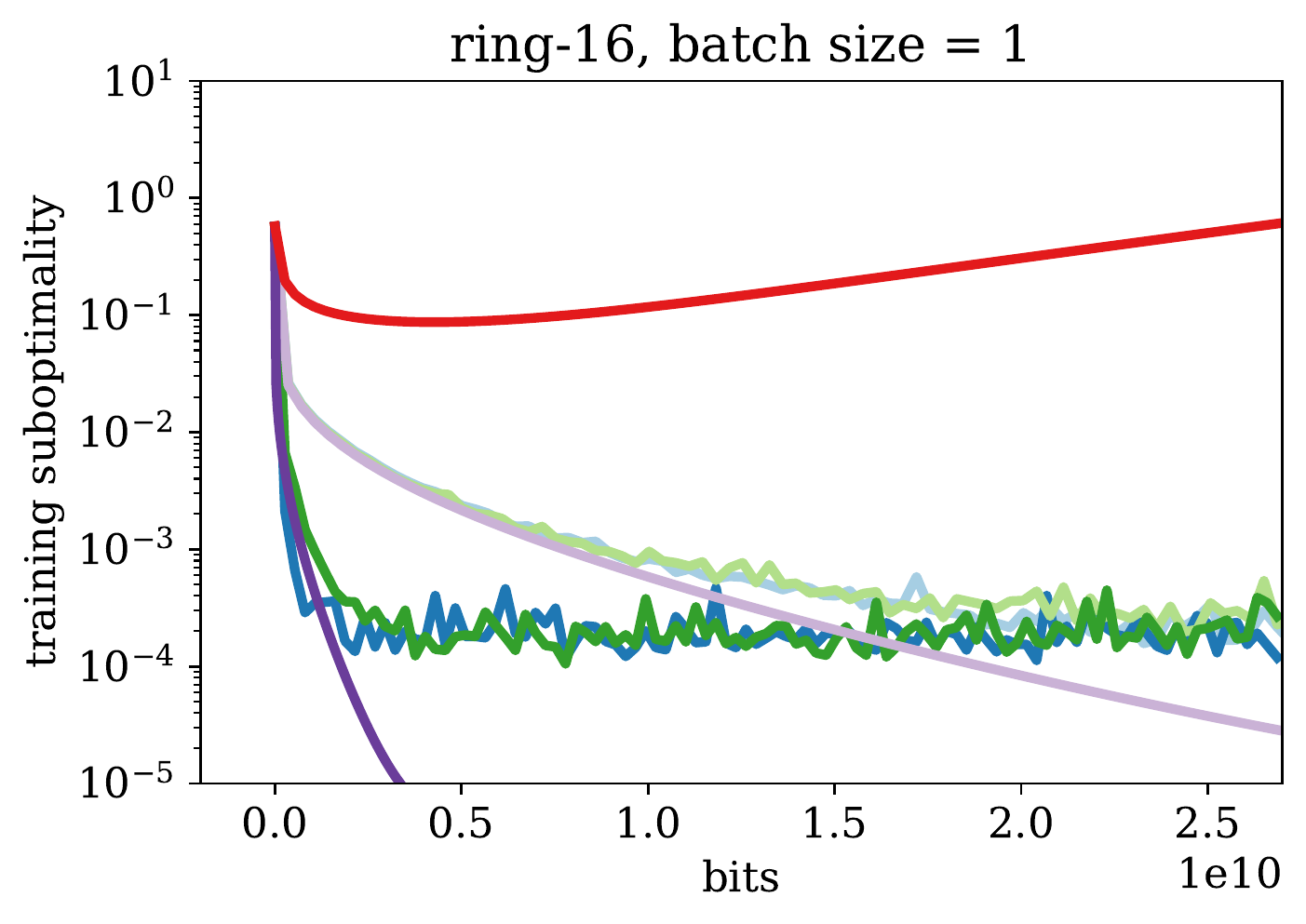}
	\hfill\null
	\vspace{-3mm}
	\caption{Logistic regression on w8a dataset. Comparison to the baselines for full batch GD (top) and stochastic GD (bottom). }\label{fig:logistic}
\end{figure*}

First, we illustrate the performance of Algorithm~\ref{alg:1} on the average consensus problem where every worker $i$ has a vector $x_i \in \R^d$ and the goal is to find the average $\bar x = \frac{1}{n}\sum_{i=1}^n x_i.$ We generate vectors $x_i$ from normal distribution $\cN(0, \mI)$. This can be cast into decentralized optimization formulation \eqref{eq:1} by considering functions of the form $f_i(x) = \smash{\frac{1}{2} \norm{x - x_i}_2^2}$. Note that these $f_i$'s are strongly convex and smooth with $L = \mu = 1$ (Assumption~\ref{assumption:convex-and-smooth}), we set $\eta = \frac{1}{L} =1$ and tune the stepsize $\theta$ for our algorithm.
In this setup we can easily compute full gradients. Moreover, both {\bf option A} and {{\bf option B}} of Algorithm~\ref{alg:1} lead to the same update. 

In Figure~\ref{fig:free} we see that that for the challenging ring topology almost any quantization level $\omega$ leads to an increase in the total number of iterations. On the other hand, as predicted by theory, for the star graph there is a level up to which quantization does not affect the convergence, and we can achieve communication savings for free. 

In Figure~\ref{fig:average_baselines} we compare our algorithms to the baselines. Even after tuning the stepsizes, QDGD converges very slowly (in agreement with Table~\ref{tab:comparison}). On both graphs, iteration-wise our algorithm converges faster than Choco. However, in terms of number of bits, Algorithm~\ref{alg:1} converges slightly slower than Choco on the ring graph. This is because our Algorithm~\ref{alg:1} requires two rounds of communication each iteration and Choco only one. However, even with this slight disadvantage, our algorithm performs best on the star graph in term of bits.

\paragraph{Logistic regression.} We further assess performance on logistic regression with the objective function $$f(x) = \frac{1}{m} \sum_{i = 1}^m \log(1 + \exp(-b_j a_j^\top x)) + \frac{1}{2m} \norm{x}_2^2,$$ where $a_j \in \R^d$, $b_j \in \{-1, 1\}$. We use the w1a dataset \cite{Platt:1998w1a} and distribute the samples between machines equally in a non-iid way, sorted by label. We use ring topology with 
$n=16$ nodes. We compare two cases: either the nodes compute gradients on their full local batch (Figure \ref{fig:logistic}, top), or stochastic gradients with respect to one single (randomly selected) local data sample (bottom).
We tune all algorithms to reach best performance after 200 epochs in the full batch case and for 300 epochs in stochastic case (left). To plot performance in terms of transmitted number of bits (right), we run the algorithms longer with found parameters. \looseness=-1

With local gradients available, our algorithm converges faster than the baselines. This is supported by the theory, as we prove linear convergence for  our {\bf option B}, while all other baselines converge only sublinearly (Table~\ref{tab:comparison}). With stochastic gradients, {\bf option C} as good as the Choco baseline, while {\bf option D} outperforms all schemes (we have proven linear rate).

\section*{Acknowledgements}
We acknowledge funding from SNSF grant 200021\_175796, as well as a Google Focused Research Award.

\small
\bibliographystyle{myplainnat}

\bibliography{../papers_diana}
\normalsize

\appendix
\onecolumn

\section{Parameters for the Numerical Experiments}

In this section we give the hyperparameters we used for the for the experiments in the paper (found by grid search).

\begin{table}[h!]
	\centering
	\begin{tabular}{l|c||c}
		&  ring-100 & {star-100}\\\hline
		method & $\theta$ & $\theta$ \\ \hline\hline
		exact (no quantization) & 1.26 & 1.58\\ \hline
		dit-$17$ & 1.26   &  1.58 \\
		dit-$9$  &  1.26 &      1.58 \\
		dit-$5$  &  1.26 &      1.58 \\
		dit-$4$  &  1.26  &      1.58 \\
		dit-$3$  &  0.79 &      1.58 \\
		dit-$2$  &   0.4 &      1.58 \\ 
		dit-$1$  &  -  &      1.58 \\\hline
		rand-$100$  &   0.8 &      1.58 \\ 
		rand-$50$  &   0.4 &      1.58 \\ 
		rand-$30$  &   0.25 &      1.58 \\ 
		rand-$20$  &   0.2 &      1.58 \\ 
		rand-$10$  &   0.1 &      1.0  \\ 
		
	\end{tabular}
	\caption{Hyperparameters found by tuning (lowest iteration number to reach target accuracy) in experiments for Fig. \ref{fig:free}. }
	\label{tab:learning_rates_fig1}
\end{table}

\begin{table}[h!]
	\centering
	\begin{tabular}{l|c||c}
		&  ring-100 & {star-100}\\\hline
		method & $\theta$ & $\theta$  \\ \hline\hline
		primal dual & 1.26 & 1.58\\
		option A/B, dit-$5$ & 1.26   &  1.58 \\
		option A/B, dit-$2$ & 0.4   &  1.58 \\\hline\hline
		method &  $\gamma$  &  $\gamma$ \\ \hline\hline
		choco, dit-$5$ & 1   &  1 \\
		choco, dit-$2$ & 0.25  &  0.2 \\ \hline\hline
		method &  $\epsilon$, $\alpha$  &  $\epsilon$, $\alpha$ \\ \hline\hline
		qdgd, dit-$5$ & (0.001, 1.0)  &  (0.0001, 10000.0) \\
		qdgd, dit-$2$ & (0.001, 1.0)  &  (0.0001, 10000.0) \\
	\end{tabular}
	\caption{Hyperparameters found by tuning in experiments for Fig. \ref{fig:average_baselines}. }
	\label{tab:learning_rates_fig2}
\end{table}

\begin{table}[h!]
	\centering
	\begin{tabular}{l|c|c}
		 \multicolumn{3}{c}{batch size = 1}    \\\hline
		method & $\gamma$  & $\eta$  \\ \hline\hline
		D-GD & -  & 0.1 \\
		choco, dit-$3$ & 0.316   & 0.1 \\ \hline\hline
		method & $\theta$  & $\eta$  \\ \hline\hline
		option C, exact &  0.1  &  0.1 \\
		option C, dit-$3$ & 3.16 $\times 10^{-3}$   &  0.1 \\
		option D, exact & 3.16  &  0.1 \\
		option D, dit-$3$ & 3.16 $\times 10^{-3}$  &  0.1  \\ \hline\hline
		method & $\epsilon$  & $\alpha$  \\ \hline\hline
		qdgd, dit-$3$ & 1e-06  &  100 \\ 
	\end{tabular}
\hspace*{1cm}
	\begin{tabular}{l|c|c}
	\multicolumn{3}{c}{full batch}\\ \hline
	method & $\gamma$  & $\eta$  \\ \hline\hline
	D-GD & -  & 21.5 \\
	choco, dit-$3$ & 0.316   & 4.64 \\\hline\hline
	method & $\theta$  & $\eta$  \\ \hline\hline
	option B, exact &  0.1  &  21.5 \\
	option B, dit-$3$ & 3.5 $\times 10^{-4}$   &  40 \\\hline\hline
	method & $\epsilon$  & $\alpha$  \\ \hline\hline
	qdgd, dit-$3$ & 0.01  &  31.6 \\ 
\end{tabular}

	\caption{Hyperparameters found by tuning (lowest error after 200 epochs for full batch, and 300 epochs for batch size 1) in experiments for Fig. \ref{fig:logistic}. }
	\label{tab:learning_rates_fig3}
\end{table}

\section{Convergence Proof}

	To prove the convergence of Algorithm~\ref{alg:1}, we will use matrix notation for the iterates of the algorithm:\\
	\begin{itemize}
	\item 	primal iterates $$\mX^k = \left[x_1^k,\ldots,x_n^k\right] \in \R^{d\times n},$$ 
	\item 
	and dual iterates $$\mZ^k = \left[z_1^k,\ldots,z_n^k\right] \in \R^{d\times n}$$
and
	$$\mH^k = \left[h_1^k,\ldots,h_n^k\right] \in \R^{d\times n}.$$
	\end{itemize}
	
	Let $\mX^\star =  [\underbrace{x^\star,\ldots, x^\star}_{n \text{ times}}]$, $\mZ^\star = \left[z_1^\star,\ldots,z_n^\star\right]$, where $x^\star$ is a solution of \eqref{eq:1}, $z_i^\star = \nabla f_i(x^\star)$ for all $i = 1,\ldots,n$.
	
	For arbitrary matrix $\mB \in \sym^n_+$, we define matrix semi-norm (norm in case $\mB$ is positive-definite)
	$\norm{\cdot}_{\mB} : \R^{d\times n} \rightarrow \R_+,$
	which is defined as follows:
	\begin{equation}
	\norm{\mX}_\mB^2 = \dotprod{\mX\mB}{\mX} = \trace{\mX\mB\mX^\top}.
	\end{equation}

	\begin{lemma}[Properties of $\mW$]
	$\ker \mW = \myspan{\1}$, where $\1 =(1,\dots,1)^\top$. 
	\end{lemma}

	\begin{lemma}
		For arbitrary $\mX \in \R^{d\times n}$, the following inequalities hold:
		\begin{equation}\label{eq:5}
			\norm{\mX}_\mW^2 \leq \lambda_{\max}(\mW) \norm{\mX}_\mI^2,
		\end{equation}
		\begin{equation}\label{eq:6}
			 \lambda_{\min}^+(\mW)\norm{\mX}_{\mW^\dagger}^2 \leq \norm{\mX}_\mI^2.
		\end{equation}
	\end{lemma}
	
	\begin{lemma}[$h_i^k$ update]
		Let $\alpha = \frac{1}{\omega + 1}$. The following inequality holds:
		\begin{equation}\label{eq:7}
		\E{\norm{\mH^{k+1} - \mX^\star}_\mI^2}
		\leq
		(1 - \alpha)\norm{\mH^{k} - \mX^\star}_\mI^2
		+
		\alpha\E{\norm{\mX^{k+1} - \mX^\star}_\mI^2}.
		\end{equation}
	\end{lemma}

	\begin{proof}
		Since $h_i^{k+1}= h_i^k + \alpha \cQ(x_i^{k+1} - h_i^k)$, we can decompose
		\begin{align*}
		\E{\norm{\mH^{k+1} - \mX^\star}_\mI^2} &= 
		\E{\norm{\mH^{k} - \mX^\star}_\mI^2} + 
		2 \alpha \dotprod{\mX^{k+1} - \mH^k}{\mH^k - \mX^\star} + \alpha^2 \E { \norm{\mX^{k+1} - \mH^k}_\mI^2 } \\
		&\leq \E{\norm{\mH^{k} - \mX^\star}_\mI^2} + 
		2 \alpha \dotprod{\mX^{k+1} - \mH^k}{\mH^k - \mX^\star} + \alpha^2 (1+\omega) \norm{\mX^{k+1} - \mH^k}_\mI^2 \\
		&\leq \E{\norm{\mH^{k} - \mX^\star}_\mI^2} + 
		2 \alpha \dotprod{\mX^{k+1} - \mH^k}{\mH^k - \mX^\star} + \alpha \norm{\mX^{k+1} - \mH^k}_\mI^2\\
		&= \E{\norm{\mH^{k} - \mX^\star}_\mI^2}  + \alpha \dotprod{\mX^{k+1} - \mH^k}{\mX^{k+1} + \mH^k - 2\mX^\star} \\
		&= \E{\norm{\mH^{k} - \mX^\star}_\mI^2} + \alpha \norm{\mX^{k+1}-\mX^\star}_\mI^2 - \alpha \norm{\mH^{k} - \mX^\star}_\mI^2 \\
		&\leq (1 - \alpha)\norm{\mH^{k} - \mX^\star}_\mI^2 + \alpha\E{\norm{\mX^{k+1} - \mX^\star}_\mI^2} .\qedhere
		\end{align*}
	\end{proof}

	\begin{lemma}[Dual step]
		\begin{equation}\label{eq:2}
			\E{\norm{\mZ^{k+1} - \mZ^\star}_{\mW^{\dagger}}^2}
			\leq
			\norm{\mZ^k - \mZ^\star}_{\mW^{\dagger}}^2
			+
			\E{-
				2\theta\dotprod{\mX^{k+1} - \mX^\star}{\mZ^k - \mZ^\star}
				+
				\theta^2
				\norm{\mX^{k+1} - \mX^\star}_\mW^2
				+
				\Sigma^k
			},
		\end{equation}
		where $\Sigma^k$ is a variance:
		\begin{equation}
			\Sigma^k = \ED{Q}{\norm{\mZ^{k+1} - \ED{Q}{\mZ^{k+1}}}_{\mW^{\dagger}}^2}.
		\end{equation}
	\end{lemma}

	\begin{proof}
		Firstly, we prove that 
		\begin{equation}\label{eq:9}
		\sum_{i=1}^n z_i^k = 0
		\end{equation}
		for all $k=0,1,2,\ldots$ by induction. For $k = 0$, \eqref{eq:9} follows trivially from initialization step of Algorithm~\ref{alg:1}.
		Now, assuming that \eqref{eq:9} holds, we have
		\begin{align*}
		\sum_{i=1}^n z_i^{k+1}
		&=
		\sum_{i=1}^n z_i^k - \theta\sum_{j \in \cN_i}w_{ij}(\Delta_{ij}^k - \Delta_{ji}^k)
		=
		-\theta\sum_{(i,j) \in E}w_{ij}\Delta_{ij}^k
		+\theta\sum_{(i,j) \in E}w_{ij}\Delta_{ji}^k\\
		&=
		-\theta\sum_{(i,j) \in E}w_{ij}\Delta_{ij}^k
		+\theta\sum_{(i,j) \in E}w_{ji}\Delta_{ij}^k
		=
		\theta\sum_{(i,j) \in E}\Delta_{ij}^k(w_{ji} - w_{ij}) = 0,
		\end{align*}
		which proves \eqref{eq:9} for all $k=0,1,2,\ldots$.
		
		Next, we show that
		\begin{equation}\label{eq:11}
		(\mZ^k - \mZ^\star)\mW^\dagger\mW = \mZ^k - \mZ^\star,
		\end{equation}
		which holds, since $(\mZ^k - \mZ^\star)\ones = 0$, where $\ones = (1,\ldots,1)^\top \in \R^n$, and hence rows of matrix $\mZ^k - \mZ^\star$ lie in range of $\mW$.
		
		Now, we rewrite $\ED{Q}{\norm{\mZ^{k+1} - \mZ^\star}_{\mW^{\dagger}}^2}$:
		\begin{equation}\label{eq:10}
			\ED{Q}{\norm{\mZ^{k+1} - \mZ^\star}_{\mW^{\dagger}}^2}
			=
			\norm{\mZ^k - \mZ^\star}_{\mW^{\dagger}}^2
			+
			\ED{Q}{2\dotprod{\mZ^{k+1} - \mZ^k}{(\mZ^k - \mZ^\star)\mW^\dagger}}
			+
			\ED{Q}{\norm{\mZ^{k+1} - \mZ^k}_{\mW^\dagger}^2}.
		\end{equation}
		To simplify the second term in \eqref{eq:10}, we first rewrite $\ED{Q}{z_i^{k+1} - z_i^k}$ as follows:
		\begin{align*}
			\ED{Q}{z_i^{k+1} - z_i^k}
			&=
			-\theta \sum_{j \in \cN_i} w_{ij}\E{\Delta_{ij}^k - \Delta_{ji}^k}
			=
			-\theta \sum_{j \in \cN_i} w_{ij}  (x_i^{k+1} - x^\star)
			+\theta \sum_{j \in \cN_i} w_{ij}  (x_j^{k+1} - x^\star)\\
			&=
			-\theta\sum_{j=1}^n \mW_{ij}(x_j^{k+1} - x^\star),
		\end{align*}
		which gives
		\begin{equation}\label{eq:13}
			\ED{Q}{\mZ^{k+1} - \mZ^k} = -\theta(\mX^{k+1} - \mX^\star)\mW.
		\end{equation}
		Hence,
		\begin{equation}\label{eq:12}
			\E{2\dotprod{\mZ^{k+1} - \mZ^k}{(\mZ^k - \mZ^\star)\mW^\dagger}}
			\overset{\eqref{eq:13}}{=}
			-2\theta\dotprod{\mX^{k+1} - \mX^\star}{(\mZ^k - \mZ^\star)\mW^\dagger\mW }
			\overset{\eqref{eq:11}}{=}
			-2\theta\dotprod{\mX^{k+1} - \mX^\star}{\mZ^k - \mZ^\star}.
		\end{equation}
		Now, we simplify the last term in \eqref{eq:10}:
		\begin{align*}
			\ED{Q}{\norm{\mZ^{k+1} - \mZ^k}_{\mW^\dagger}^2}
			&=
			\norm{\ED{Q}{\mZ^{k+1} - \mZ^k}}_{\mW^\dagger}^2
			+
			\ED{Q}{\norm{\mZ^{k+1} - \ED{Q}{\mZ^{k+1}}}_{\mW^\dagger}^2}\\
			&\overset{\eqref{eq:13}}{=}
			\theta^2\norm{(\mX^{k+1} - \mX^\star)\mW}_{\mW^\dagger}^2+\Sigma^k
			=
			\theta^2\norm{\mX^{k+1} - \mX^\star}_{\mW}^2+\Sigma^k.
		\end{align*}
		Plugging this and \eqref{eq:12} into \eqref{eq:10} gives
		\begin{align*}
			\ED{Q}{\norm{\mZ^{k+1} - \mZ^\star}_{\mW^{\dagger}}^2}
			\leq
			\norm{\mZ^k - \mZ^\star}_{\mW^{\dagger}}^2
			-
			2\theta\dotprod{\mX^{k+1} - \mX^\star}{\mZ^k - \mZ^\star}
			+
			\theta^2
			\norm{\mX^{k+1} - \mX^\star}_\mW^2
			+
			\Sigma^k.
		\end{align*}
		Taking full expectation concludes the proof.
	\end{proof}

	\begin{lemma}[Variance bound] For each $k\geq 0$ we have
		\begin{equation}\label{eq:3}
			\Sigma^k \leq 4\theta^2\omega \max_{(i,j)\in E}w_{ij}
			\left[
				\norm{\mX^{k+1} - \mX_\star}_{\mI}^2
				+
				\norm{\mH^{k} - \mX_\star}_{\mI}^2
			\right].
		\end{equation}
	\end{lemma}

	\begin{proof}
		\begin{align}
			\Sigma^k
			&=
			\ED{Q}{\norm{\mZ^{k+1} - \ED{Q}{\mZ^{k+1}}}_{\mW^{\dagger}}^2}
			=
			\EE_Q
				\sum_{i=1}^n \sum_{j=1}^n
				\mW^\dagger_{ij}
				\dotprod{z_i^{k+1} - \EE_Q z_i^{k+1}}{z_j^{k+1} - \EE_Q z_j^{k+1}}\\
			&=
			\sum_{i=1}^n \sum_{j=1}^n
			\mW^\dagger_{ij}
			\ED{Q}{
			\dotprod{z_i^{k+1} - \EE_Q z_i^{k+1}}{z_j^{k+1} - \EE_Q z_j^{k+1}}},\label{eq:15}
		\end{align}
		where $\mW^\dagger_{ij}$ is number in the $i$-th row of $j$-th column of $\mW^\dagger$.
		
		Next we observe, that
		\begin{equation}\label{eq:14}
			z_i^{k+1} - \EE_Q z_i^{k+1} = -\theta\sum_{j\in \cN_i} w_{ij}
			(\bar{\Delta}_{ij}^k - \bar{\Delta}_{ji}^k),
		\end{equation}
		where
		$
			\bar{\Delta}_{ij}^k := \Delta_{ij}^k - \EE_Q \Delta_{ij} = \Delta_{ij}^k - x_i^{k+1}.
		$
		
		From the construction of $\Delta_{ij}^k$ it follows that
		\begin{equation}
			\ED{Q}{\dotprod{\bar{\Delta}_{ab}^k}{\bar{\Delta}_{cd}^k}} = 
			\begin{cases}
				\sigma_a^k,& a=c \text{ and } b =d\\
				0,& \text{otherwise}
			\end{cases}
			=
			\delta_{ac}\delta_{bd} \sigma_a^k,
		\end{equation}
		where
		$
			\sigma_i^k = \ED{Q}{\norm{Q(x_i^{k+1} - h_i^k) - (x_i^{k+1} - h_i^k)}_2^2}
		$
		and $\delta_{ij}$
		is a Kronecker delta.
		
		Now, we rewrite $\ED{Q}{
			\dotprod{z_i^{k+1} - \EE_Q z_i^{k+1}}{z_j^{k+1} - \EE_Q z_j^{k+1}}}$:
		\begin{eqnarray*}
			\ED{Q}{
				\dotprod{z_i^{k+1} - \EE_Q z_i^{k+1}}{z_j^{k+1} - \EE_Q z_j^{k+1}}}
			&\overset{\eqref{eq:14}}{=}&
			\theta^2
			\ED{Q}{
				\Dotprod{
					\sum_{p \in \cN_i}w_{ip}(\bar{\Delta}_{ip}^k - \bar{\Delta}_{pi}^k)
				}{
					\sum_{q \in \cN_i}w_{jq}(\bar{\Delta}_{jq}^k - \bar{\Delta}_{qj}^k)
				}
			}\\
			&=&
			\theta^2\ED{Q}{
				\Dotprod{
					\sum_{p \neq i}\mW_{ip}(\bar{\Delta}_{ip}^k - \bar{\Delta}_{pi}^k)
				}{
					\sum_{q \neq j}\mW_{jq}(\bar{\Delta}_{jq}^k - \bar{\Delta}_{qj}^k)
				}
			}\\
			&=&
			\theta^2\sum_{p \neq i}\sum_{q \neq j}
			\mW_{ip}\mW_{jq}
			\ED{Q}{
			\dotprod{\bar{\Delta}_{ip}^k - \bar{\Delta}_{pi}^k}{\bar{\Delta}_{jq}^k - \bar{\Delta}_{qj}^k}}\\
			&=&
			\theta^2\sum_{p \neq i}\sum_{q \neq j}
			\mW_{ip}\mW_{jq}
			\left[
				\delta_{ij}\delta_{pq}(\sigma_i^k  + \sigma_p^k)
				-
				\delta_{iq}\delta_{pj} (\sigma_i^k + \sigma_p^k)
			\right]\\
			&=&
			\theta^2\sum_{p \neq i}
			\mW_{ip}(\sigma_i^k  + \sigma_p^k)
			\sum_{q \neq j}
			\mW_{jq}
			\left[
				\delta_{ij}\delta_{pq} - \delta_{iq}\delta_{pj}
			\right]\\
			&=&
			\theta^2\delta_{ij}
			\sum_{p \neq i}
			\mW_{ip}^2(\sigma_i^k  + \sigma_p^k)
			\sum_{q \neq i}
			\delta_{pq}
			-
			\theta^2\mW_{ij}^2(\sigma_i^k  + \sigma_j^k)
			\sum_{p \neq i}
			\delta_{pj}
			\sum_{q \neq j}
			\delta_{iq}\\
			&=&
			\theta^2\delta_{ij}
			\sum_{p \neq i}
			\mW_{ip}^2(\sigma_i^k  + \sigma_p^k)
			-
			\theta^2(1-\delta_{ij})\mW_{ij}^2(\sigma_i^k  + \sigma_j^k).
		\end{eqnarray*}
	Plugging this into \eqref{eq:15} gives
	\begin{align*}
		\Sigma^k
		&=
		\theta^2
		\sum_{i=1}^n \sum_{j=1}^n
		\mW^\dagger_{ij}
		\delta_{ij}
		\sum_{p \neq i}
		\mW_{ip}^2(\sigma_i^k  + \sigma_p^k)
		-
		\theta^2
		\sum_{i=1}^n \sum_{j=1}^n
		\mW^\dagger_{ij}
		(1-\delta_{ij})\mW_{ij}^2(\sigma_i^k  + \sigma_j^k)\\
		&=
		\theta^2
		\sum_{i=1}^n 
		\mW^\dagger_{ii}
		\sum_{p \neq i}
		\mW_{ip}^2(\sigma_i^k  + \sigma_p^k)
		-
		\theta^2
		\sum_{i=1}^n \sum_{j\neq i}
		\mW^\dagger_{ij}
		\mW_{ij}^2(\sigma_i^k  + \sigma_j^k)\\
		&=
		\theta^2
		\sum_{i=1}^n 
		\mW^\dagger_{ii}
		\sum_{j \neq i}
		\mW_{ij}^2(\sigma_i^k  + \sigma_j^k)
		-
		\theta^2
		\sum_{i=1}^n \sum_{j\neq i}
		\mW^\dagger_{ij}
		\mW_{ij}^2(\sigma_i^k  + \sigma_j^k)\\
		&=
		\theta^2
		\sum_{i=1}^n \sum_{j\neq i}
		(\mW^\dagger_{ii} - \mW^\dagger_{ij})\mW_{ij}^2(\sigma_i^k  + \sigma_j^k)
		=
		\theta^2
		\sum_{i=1}^n \sum_{j\neq i}
		(\mW^\dagger_{ii} - \mW^\dagger_{ij})\mW_{ij}^2(\sigma_i^k  + \sigma_j^k)\\
		&=
		2\theta^2
		\sum_{i=1}^n \sum_{j\neq i}
		(\mW^\dagger_{ii} - \mW^\dagger_{ij})\mW_{ij}^2\sigma_i^k
		=
		2\theta^2
		\sum_{i=1}^n \sigma_i^k\sum_{j\neq i}
		(\mW^\dagger_{ii} - \mW^\dagger_{ij})\mW_{ij}^2\\
		&=
		2\theta^2\sum_{i=1}^n
		\sigma_i^k
		\left(
			\mW^\dagger_{ii}\sum_{j \neq i}\mW_{ij}^2
			-
			\sum_{j \neq i} \mW^\dagger_{ij} \mW_{ij}^2
		\right)
		=
		2\theta^2
		\sum_{i=1}^n \sigma_i^k
		\sum_{j=1}^n \mW^\dagger_{ij}\hat{\mW}_{ij},
	\end{align*}
	where $\hat{\mW}$ is another Laplacian matrix:
	\begin{equation}
		\hat{\mW}_{ij} = \begin{cases}
		-\mW_{ij}^2, & i \neq j\\
		\sum_{l \neq j} \mW_{il}^2, & i=j
		\end{cases}.
	\end{equation}
	This gives us
	\begin{align*}
		\Sigma^k
		&=
		2\theta^2
		\sum_{i=1}^n \sigma_i^k
		\sum_{j=1}^n \mW^\dagger_{ij}\hat{\mW}_{ij}
		=
		2\theta^2
		\sum_{i=1}^n \sigma_i^k
		[\mW^\dagger \hat{\mW}]_{ii}
		=
		2\theta^2
		\sum_{i=1}^n \sigma_i^k
		[\mW^{\dagger/2} \hat{\mW} \mW^{\dagger/2}]_{ii}\\
		&\leq
		2\theta^2\lambda_{\max}(\mW^{\dagger/2} \hat{\mW} \mW^{\dagger/2})
		\sum_{i=1}^n \sigma_i^k.
	\end{align*}
	Now, we use the definition of $\sigma_i^k$ and get
	\begin{align}
		\Sigma^k
		&\leq
		2\theta^2\lambda_{\max}(\mW^{\dagger/2} \hat{\mW} \mW^{\dagger/2})
		\sum_{i=1}^n \ED{Q}{\norm{Q(x_i^{k+1} - h_i^k) - (x_i^{k+1} - h_i^k)}_2^2}\\
		&\leq
		2\theta^2\omega\lambda_{\max}(\mW^{\dagger/2} \hat{\mW} \mW^{\dagger/2})
		\sum_{i=1}^n \omega\norm{x_i^{k+1} - h_i^k}_2^2
		=
		2\theta^2\omega\lambda_{\max}(\mW^{\dagger/2} \hat{\mW} \mW^{\dagger/2})
		\norm{\mX^{k+1} - \mH^k}_{\mI}^2\\
		&\leq
		4\theta^2\omega\lambda_{\max}(\mW^{\dagger/2} \hat{\mW} \mW^{\dagger/2})
		\left[
		\norm{\mX^{k+1} - \mX_\star}_{\mI}^2
		+
		\norm{\mH^{k} - \mX_\star}_{\mI}^2
		\right].\label{eq:16}
	\end{align}
	
	It remains to upper-bound $\lambda_{\max}(\mW^{\dagger/2} \hat{\mW} \mW^{\dagger/2})$. We note that for all $u \in \R^n$
	\begin{align*}
		u^\top \hat{\mW} u
		&=
		\sum_{i=1}^n\sum_{j=1}^n
			\hat{\mW}_{ij} u_iu_j
		=
		\sum_{i=1}^n u_i^2 \sum_{j \neq i} \mW_{ij}^2
		-
		\sum_{i=1}^n\sum_{j \neq i}u_iu_j \mW_{ij}^2\\
		&=
		\frac{1}{2}\sum_{i=1}^n\sum_{j \neq i} (u_i^2 + u_j^2) \mW_{ij}^2
		-
		\sum_{i=1}^n\sum_{j \neq i}u_iu_j \mW_{ij}^2
		=
		\frac{1}{2}\sum_{i=1}^n\sum_{j \neq i} (u_i - u_j)^2 \mW_{ij}^2\\
		&\leq
		\frac{1}{2}
		\max_{(i,j)\in E} w_{ij}
		\sum_{i=1}^n\sum_{j \neq i} (u_i - u_j)^2 \mW_{ij}
		=
		\max_{(i,j)\in E} w_{ij} \cdot u^\top \mW u.
	\end{align*} 
	Hence, for all $u \in \R^n$
	\begin{align*}
		u^\top \mW^{\dagger/2} \hat{\mW} \mW^{\dagger/2} u
		&\leq
		\max_{(i,j)\in E} w_{ij} \cdot
		u^\top \mW^{\dagger/2} \mW \mW^{\dagger/2} u
		=
		\max_{(i,j)\in E} w_{ij} \cdot
		u^\top \mW^{\dagger} \mW u
		\leq
		\max_{(i,j)\in E} w_{ij} \cdot \norm{u}_2^2
	\end{align*}
	and thus $\lambda_{\max}(\mW^{\dagger/2} \hat{\mW} \mW^{\dagger/2})$ is bounded by
	\begin{equation}
		\lambda_{\max}(\mW^{\dagger/2} \hat{\mW} \mW^{\dagger/2})
		\leq
		\max_{(i,j)\in E} w_{ij}.
	\end{equation}
	plugging this into \eqref{eq:16} gives \eqref{eq:3} and concludes the proof.
	
	\end{proof}

	\subsection{Option A}

	\begin{lemma}[Primal step, Option A]
		The following inequality holds:
		\begin{equation}\label{eq:4}
			-2\theta\dotprod{\mX^{k+1} - \mX^\star}{\mZ^k - \mZ^\star} 
			\leq
			-\frac{\theta}{L}\norm{\mZ^k - \mZ^\star}_\mI^2
			-
			\theta\mu\norm{\mX^{k+1} - \mX^\star}_\mI^2.
		\end{equation}
	\end{lemma}

	\begin{proof}
		From Line~\ref{line:1} of Algorithm~\ref{alg:1} it follows that
		\begin{align*}
			-2\theta\dotprod{\mX^{k+1} - \mX^\star}{\mZ^k - \mZ^\star}
			&=
			-2\theta \sum_{i=1}^n \dotprod{\nabla f_i^\star(z_i^k) - \nabla f_i^\star(z_i^\star)}{z_i^k - z_i^\star}\\
			&\leq
			-\frac{\theta}{L}\sum_{i=1}^n \norm{z_i^k - z_i^\star}_2^2
			-\theta\mu \sum_{i=1}^n \norm{\nabla f_i^*(z_i^k) - \nabla f_i^*(z_i^\star)}_2^2\\
			&=
			-\frac{\theta}{L}\norm{\mZ^k - \mZ^\star}_\mI^2
			-
			\theta\mu\norm{\mX^{k+1} - \mX^\star}_\mI^2,
		\end{align*}
		where we used strong monotonicity and co-coercivity of $\nabla f_i^*$ in the inequality. %
	\end{proof}
	
	\begin{theorem}[Convergence of Algorithm~\ref{alg:1}, Option A]\label{thm:A}
		Let $\Psi_A^k$ be a Lyapunov function which is defined as follows:
		\begin{equation}
			\Psi_A^k
			= 
			\norm{\mZ^{k} - \mZ^\star}_{\mW^{\dagger}}^2
			+
			\frac{8\theta^2\omega \max_{(i,j)\in E}w_{ij}}{\alpha}
			\norm{\mH^{k} - \mX^\star}_\mI^2.
		\end{equation}
		Let $\rho_A$ be defined as follows:
		\begin{equation}
		\rho_A = \max \left\{2(\omega + 1), \frac{L(\lambda_{\max}(\mW) + 12\omega\max_{(i,j)\in E}w_{ij})}{\mu \lambda_{\min}^+(\mW)}\right\}^{-1}.
		\end{equation}
		Choosing the stepsize $\theta$ as
		\begin{equation}\label{eq:8}
			\theta = \frac{\mu}{\lambda_{\max}(\mW) + 12\omega\max_{(i,j)\in E}w_{ij}}
		\end{equation}
		and stepsize $\alpha = \frac{1}{\omega + 1}$ gives the following inequality:
		\begin{equation}
			\E{
				\Psi_A^{k+1}
			}
			\leq
			(1-\rho_A)\Psi_A^k.
		\end{equation}
	\end{theorem}

	\begin{proof}
		We start with rewriting \eqref{eq:2}:
		\begin{eqnarray*}
		\E{\norm{\mZ^{k+1} - \mZ^\star}_{\mW^{\dagger}}^2}
		&\overset{\eqref{eq:2}}{\leq}&
		\norm{\mZ^k - \mZ^\star}_{\mW^{\dagger}}^2
		+
		\E{-
			2\theta\dotprod{\mX^{k+1} - \mX^\star}{\mZ^k - \mZ^\star}
			+
			\theta^2
			\norm{\mX^{k+1} - \mX^\star}_\mW^2
			+
			\Sigma^k
		}\\
		&\overset{\eqref{eq:4}}{\leq}&
		\norm{\mZ^k - \mZ^\star}_{\mW^{\dagger}}^2
		-
		\frac{\theta}{L}\norm{\mZ^k - \mZ^\star}_\mI^2
		-
		\theta\mu\norm{\mX^{k+1} - \mX^\star}_\mI^2
		+
		\theta^2\norm{\mX^{k+1} - \mX^\star}_\mW^2
		+
		\E{\Sigma^k}\\
		&\overset{\eqref{eq:5},\eqref{eq:6}}{\leq}&
		\left(
			1 - \frac{\theta\lambda_{\min}^+(\mW)}{L}
		\right)
		\norm{\mZ^k - \mZ^\star}_{\mW^{\dagger}}^2
		-
		\theta(\mu - \theta\lambda_{\max}(\mW))\norm{\mX^{k+1} - \mX^\star}_\mI^2
		+
		\E{\Sigma^k}\\
		&\overset{\eqref{eq:3}}{\leq}&
		\left(
		1 - \frac{\theta\lambda_{\min}^+(\mW)}{L}
		\right)
		\norm{\mZ^k - \mZ^\star}_{\mW^{\dagger}}^2
		+
		4\theta^2\omega \max_{(i,j)\in E}w_{ij}\norm{\mH^{k} - \mX_\star}_{\mI}^2\\
		&-&
		\theta\left(
			\mu - \theta\left[\lambda_{\max}(\mW) + 4\omega\max_{(i,j)\in E}w_{ij} \right]
		\right)\norm{\mX^{k+1} - \mX^\star}_\mI^2.
		\end{eqnarray*}
		Now, we combine this with \eqref{eq:7}:
		\begin{align*}
		\E{
			\Psi_A^{k+1}
		}
		&\leq
			\left(
		1 - \frac{\theta\lambda_{\min}^+(\mW)}{L}
		\right)
		\norm{\mZ^k - \mZ^\star}_{\mW^{\dagger}}^2
		+
		4\theta^2\omega \max_{(i,j)\in E}w_{ij}\norm{\mH^{k} - \mX_\star}_{\mI}^2\\
		&-
		\theta\left(
		\mu - \theta\left[\lambda_{\max}(\mW) + 4\omega\max_{(i,j)\in E}w_{ij} \right]
		\right)\norm{\mX^{k+1} - \mX^\star}_\mI^2\\
		&+
		(1-\alpha)\frac{8\theta^2\omega \max_{(i,j)\in E}w_{ij}}{\alpha}
		\norm{\mH^{k} - \mX^\star}_\mI^2
		+
		8\theta^2\omega \max_{(i,j)\in E}w_{ij}
		\norm{\mX^{k+1} - \mX^\star}_\mI^2\\
		&=
		\left(
		1 - \frac{\theta\lambda_{\min}^+(\mW)}{L}
		\right)\norm{\mZ^k - \mZ^\star}_{\mW^{\dagger}}^2
		+
		\left(1 - \frac{\alpha}{2}\right)\frac{8\theta^2\omega \max_{(i,j)\in E}w_{ij}}{\alpha}
		\norm{\mH^{k} - \mX^\star}_\mI^2\\
		&+
		\theta\left(
		\mu - \theta\left[\lambda_{\max}(\mW) + 12\omega\max_{(i,j)\in E}w_{ij} \right]
		\right)\norm{\mX^{k+1} - \mX^\star}_\mI^2.
		\end{align*}
		Using \eqref{eq:8} we get
		\begin{align*}
		\E{
			\Psi_A^{k+1}
		}
		&\leq
		\left(
		1 - \frac{\theta\lambda_{\min}^+(\mW)}{L}
		\right)\norm{\mZ^k - \mZ^\star}_{\mW^{\dagger}}^2
		+
		\left(1 - \frac{\alpha}{2}\right)\frac{8\theta^2\omega \max_{(i,j)\in E}w_{ij}}{\alpha}
		\norm{\mH^{k} - \mX^\star}_\mI^2\\
		&\leq
		(1-\rho_A)\Psi_A^k,
		\end{align*}
		which concludes the proof.
	\end{proof}

	\subsection{Option B}

	\begin{lemma}[Primal step, Option B]
		Let $\eta \leq \frac{1}{L}$. Then the following inequality holds:
		\begin{align}
		-2\theta \dotprod{\mX^{k+1} - \mX^\star}{\mZ^k - \mZ^\star}
		&\leq
		-\eta \theta \norm{\mZ^k - \mZ^\star}_\mI^2
		-
		\frac{\theta\mu}{2}\norm{\mX^{k+1} - \mX^\star}_\mI^2\nonumber\\
		&+
		(1-\eta\mu)\frac{\theta}{\eta}\norm{\mX^k - \mX^\star}_\mI^2
		-
		\left(1 - \frac{\eta\mu}{2}\right)\frac{\theta}{\eta}\norm{\mX^{k+1} - \mX^\star}_\mI^2.\label{eq:17}
		\end{align}
	\end{lemma}

	\begin{proof}
		From Line~\ref{line:2} of Algorithm~\ref{alg:1} it follows that
		\begin{align*}
			\norm{x_i^{k+1} - x^\star - \eta(z_i^k - z_i^\star)}_2^2 &= \norm{x_i^k - x^\star - \eta (\nabla f_i(x_i^k) - \nabla f_i(x^\star)) }_2^2\\
			&\leq (1-\eta\mu)\norm{x_i^k - x_i^\star}_2^2
		\end{align*}
		for any stepsize $\eta \leq \frac{1}{L}$, which leads to
		\begin{align*}
			(1-\eta\mu)\norm{\mX^k - \mX^\star}_\mI^2
			&= 
			\sum_{i=1}^n  (1-\eta\mu)\norm{x_i^k - x_i^\star}_2^2
			\geq
			\sum_{i=1}^n  \norm{x_i^{k+1} - x^\star - \eta(z_i^k - z_i^\star)}_2^2
			\\
			&=
			\norm{\mX^{k+1} - \mX^\star - \eta(\mZ^k - \mZ^\star)}_{\mI}^2\\
			&=
			\norm{\mX^{k+1} - \mX^\star}_\mI^2 + \eta^2 \norm{\mZ^k - \mZ^\star}_\mI^2
			-
			2\eta\dotprod{\mX^{k+1} - \mX^\star}{\mZ^k - \mZ^\star}.
		\end{align*}
		After rearranging, we get
		\begin{align*}
			-2\theta \dotprod{\mX^{k+1} - \mX^\star}{\mZ^k - \mZ^\star}
			&\leq
			-\eta \theta \norm{\mZ^k - \mZ^\star}_\mI^2
			-
			\frac{\theta\mu}{2}\norm{\mX^{k+1} - \mX^\star}_\mI^2\\
			&+
			(1-\eta\mu)\frac{\theta}{\eta}\norm{\mX^k - \mX^\star}_\mI^2
			-
			\left(1 - \frac{\eta\mu}{2}\right)\frac{\theta}{\eta}\norm{\mX^{k+1} - \mX^\star}_\mI^2,
		\end{align*}
		which concludes the proof.
		
	\end{proof}

	\begin{theorem}[Convergence of Algorithm~\ref{alg:1}, Option B]\label{thm:B}
		Let $\Psi_B^k$ be a Lyapunov function defined as follows:
		\begin{equation}
			\Psi_B^k = \norm{\mZ^k - \mZ^\star}_{\mW^\dagger}^2 + \frac{(1 - \eta\mu/2)\theta}{\eta}\norm{\mX^k - \mX^\star}_\mI^2 + \frac{8\theta^2\omega \max_{(i,j)\in E}w_{ij}}{\alpha}
			\norm{\mH^{k} - \mX^\star}_\mI^2.
		\end{equation}
		Let $\rho_B$ be defined as follows:
		\begin{equation}
		\rho_B = \max \left\{2(\omega + 1), \frac{2L(\lambda_{\max}(\mW) + 12\omega\max_{(i,j)\in E}w_{ij})}{\mu \lambda_{\min}^+(\mW)}\right\}^{-1}.
		\end{equation}
		Choosing stepsize $\theta$
		\begin{equation}\label{eq:18}
		\theta = \frac{\mu}{2\lambda_{\max}(\mW) + 24\omega\max_{(i,j)\in E}w_{ij}},
		\end{equation}
		stepsize $\eta = \frac{1}{L}$ and stepsize $\alpha = \frac{1}{\omega + 1}$ gives the following inequality:
		\begin{equation}
			\E{\Psi_B^{k+1}} \leq (1-\rho_B)\Psi_B^k.
		\end{equation}
		
	\end{theorem}

	\begin{proof}
		We start with rewriting \eqref{eq:2}:
		\begin{eqnarray*}
		\E{\norm{\mZ^{k+1} - \mZ^\star}_{\mW^{\dagger}}^2}
		&\overset{\eqref{eq:2}}{\leq}&
		\norm{\mZ^k - \mZ^\star}_{\mW^{\dagger}}^2
		+
		\E{-
			2\theta\dotprod{\mX^{k+1} - \mX^\star}{\mZ^k - \mZ^\star}
			+
			\theta^2
			\norm{\mX^{k+1} - \mX^\star}_\mW^2
			+
			\Sigma^k
		}\\
		&\overset{\eqref{eq:17}}{\leq}&
		\norm{\mZ^k - \mZ^\star}_{\mW^{\dagger}}^2
		-
		\eta\theta\norm{\mZ^k - \mZ^\star}_\mI^2
		-
		\frac{\theta\mu}{2}\norm{\mX^{k+1} - \mX^\star}_\mI^2
		+
		\theta^2\norm{\mX^{k+1} - \mX^\star}_\mW^2
		+
		\E{\Sigma^k}\\
		&+&
		(1-\eta\mu)\frac{\theta}{\eta}\norm{\mX^k - \mX^\star}_\mI^2
		-
		\left(1 - \frac{\eta\mu}{2}\right)\frac{\theta}{\eta}\norm{\mX^{k+1} - \mX^\star}_\mI^2
		\\
		&\overset{\eqref{eq:5},\eqref{eq:6}}{\leq}&
		(
		1 - \eta\theta\lambda_{\min}^+(\mW)
		)
		\norm{\mZ^k - \mZ^\star}_{\mW^{\dagger}}^2
		-
		\theta\left(\frac{\mu}{2} - \theta\lambda_{\max}(\mW)\right)\norm{\mX^{k+1} - \mX^\star}_\mI^2
		+
		\E{\Sigma^k}\\
		&+&
		(1-\eta\mu)\frac{\theta}{\eta}\norm{\mX^k - \mX^\star}_\mI^2
		-
		\left(1 - \frac{\eta\mu}{2}\right)\frac{\theta}{\eta}\norm{\mX^{k+1} - \mX^\star}_\mI^2
		\\
		&\overset{\eqref{eq:3}}{\leq}&
		(
		1 - \eta\theta\lambda_{\min}^+(\mW)
		)
		\norm{\mZ^k - \mZ^\star}_{\mW^{\dagger}}^2
		+
		4\theta^2\omega \max_{(i,j)\in E}w_{ij}\norm{\mH^{k} - \mX_\star}_{\mI}^2\\
		&-&
		\theta\left(
		\frac{\mu}{2} - \theta\left[\lambda_{\max}(\mW) + 4\omega\max_{(i,j)\in E}w_{ij} \right]
		\right)\norm{\mX^{k+1} - \mX^\star}_\mI^2\\
		&+&
		(1-\eta\mu)\frac{\theta}{\eta}\norm{\mX^k - \mX^\star}_\mI^2
		-
		\left(1 - \frac{\eta\mu}{2}\right)\frac{\theta}{\eta}\norm{\mX^{k+1} - \mX^\star}_\mI^2.
		\end{eqnarray*}
		Now, we combine this with \eqref{eq:7}:
		\begin{align*}
			\E{\Psi_B^{k+1}}
			&\leq
			(1 - \eta\theta\lambda_{\min}^+(\mW))\norm{\mZ^k - \mZ^\star}_{\mW^{\dagger}}^2
			+
			(1-\eta\mu)\frac{\theta}{\eta}\norm{\mX^k - \mX^\star}_\mI^2
			+
			4\theta^2\omega \max_{(i,j)\in E}w_{ij}\norm{\mH^{k} - \mX_\star}_{\mI}^2\\
			&-
			\theta\left(
			\frac{\mu}{2} - \theta\left[\lambda_{\max}(\mW) + 4\omega\max_{(i,j)\in E}w_{ij} \right]
			\right)\norm{\mX^{k+1} - \mX^\star}_\mI^2\\
			&+
			(1-\alpha)\frac{8\theta^2\omega \max_{(i,j)\in E}w_{ij}}{\alpha}
			\norm{\mH^{k} - \mX^\star}_\mI^2
			+
			8\theta^2\omega \max_{(i,j)\in E}w_{ij}
			\norm{\mX^{k+1} - \mX^\star}_\mI^2\\
			&=
			(1 - \eta\theta\lambda_{\min}^+(\mW))\norm{\mZ^k - \mZ^\star}_{\mW^{\dagger}}^2
			+
			(1-\eta\mu)\frac{\theta}{\eta}\norm{\mX^k - \mX^\star}_\mI^2\\
			&+
			\left(1-\frac{\alpha}{2}\right)\frac{8\theta^2\omega \max_{(i,j)\in E}w_{ij}}{\alpha}
			\norm{\mH^{k} - \mX^\star}_\mI^2\\
			&-
			\theta\left(
			\frac{\mu}{2} - \theta\left[\lambda_{\max}(\mW) + 12\omega\max_{(i,j)\in E}w_{ij} \right]
			\right)\norm{\mX^{k+1} - \mX^\star}_\mI^2.
		\end{align*}
		Using \eqref{eq:18} we get
		\begin{align*}
			\E{\Psi_B^{k+1}}
			&\leq
			(1 - \eta\theta\lambda_{\min}^+(\mW))\norm{\mZ^k - \mZ^\star}_{\mW^{\dagger}}^2
			+
			\left(1 - \frac{\eta\mu}{2 - \eta\mu}\right)\frac{(1 - \eta\mu/2)\theta}{\eta}\norm{\mX^k - \mX^\star}_\mI^2\\
			&+
			\left(1-\frac{\alpha}{2}\right)\frac{8\theta^2\omega \max_{(i,j)\in E}w_{ij}}{\alpha}
			\norm{\mH^{k} - \mX^\star}_\mI^2\\
			&\leq
			(1 - \eta\theta\lambda_{\min}^+(\mW))\norm{\mZ^k - \mZ^\star}_{\mW^{\dagger}}^2
			+
			\left(1 - \frac{\eta\mu}{2}\right)\frac{(1 - \eta\mu/2)\theta}{\eta}\norm{\mX^k - \mX^\star}_\mI^2\\
			&+
			\left(1-\frac{\alpha}{2}\right)\frac{8\theta^2\omega \max_{(i,j)\in E}w_{ij}}{\alpha}
			\norm{\mH^{k} - \mX^\star}_\mI^2\\
			&\leq
			(1-\rho_B)\Psi_B^k,
		\end{align*}
		which concludes the proof.
	\end{proof}

	\subsection{Option C}

	\begin{lemma}[Primal step, Option C]
		Let $\eta \leq \frac{1}{2L}$. Then the following inequality holds:
		\begin{align}
		-2\theta \E{\dotprod{\mX^{k+1} - \mX^\star}{\mZ^k - \mZ^\star}}
		&\leq
		-\eta \theta \norm{\mZ^k - \mZ^\star}_\mI^2
		-
		\frac{\theta\mu}{2}\E{\norm{\mX^{k+1} - \mX^\star}_\mI^2} + 2n\eta\theta\sigma^2\\
		&+
		(1-\eta\mu)\frac{\theta}{\eta}\norm{\mX^k - \mX^\star}_\mI^2
		-
		\left(1 - \frac{\eta\mu}{2}\right)\frac{\theta}{\eta}\E{\norm{\mX^{k+1} - \mX^\star}_\mI^2}. \label{eq:24}
		\end{align}
	\end{lemma}
	
	\begin{proof}
		From Line~\ref{line:3} of Algorithm~\ref{alg:1} it follows that
		\begin{eqnarray*}
		\E{\norm{x_i^{k+1} - x^\star - \eta(z_i^k - z_i^\star)}_2^2}
		&=&
		\E{\norm{x_i^k - x^\star - \eta (\nabla f_i(x_i^k,\xi_i^k) - \nabla f_i(x^\star)) }_2^2}\\
		&=&
		\norm{x_i^k - x^\star}_2^2 - 2\eta\E{\dotprod{\nabla f_i(x_i^k,\xi_i^k) - \nabla f_i(x^\star)}{x_i^k - x^\star}}\\
		&+&
		\eta^2\E{\norm{\nabla f_i(x_i^k,\xi_i^k) - \nabla f_i(x^\star)}_2^2}\\
		&\overset{\eqref{eq:19}}{\leq}&
		\norm{x_i^k - x^\star}_2^2 - 2\eta\dotprod{\nabla f_i(x_i^k) - \nabla f_i(x^\star)}{x_i^k - x^\star}\\
		&+&
		2\eta^2\E{\norm{\nabla f_i(x_i^k,\xi_i^k) - \nabla f_i(x^\star,\xi_i^k)}_2^2 + \norm{\nabla f_i(x^\star,\xi_i^k) - \nabla f_i(x^\star)}_2^2}
		\\
		&\overset{\eqref{eq:22},\eqref{eq:20},\eqref{eq:21}}{\leq}&
		(1-\eta\mu)\norm{x_i^k - x^\star}_2^2 - 2\eta B_{f_i}(x_i^k, x_i^\star)
		+
		4L\eta^2 B_{f_i}(x_i^k,x^\star) + 2\eta^2\sigma_i^2
		\\
		&\leq&
		(1-\eta\mu)\norm{x_i^k - x_i^\star}_2^2 + 2\eta^2\sigma_i^2,
		\end{eqnarray*}
		where we used $\eta \leq \frac{1}{2L}$ in the last inequality.
		This leads to
		\begin{align*}
		(1-\eta\mu)\norm{\mX^k - \mX^\star}_\mI^2
		&= 
		\sum_{i=1}^n  (1-\eta\mu)\norm{x_i^k - x_i^\star}_2^2
		\geq
		\E{\sum_{i=1}^n  \norm{x_i^{k+1} - x^\star - \eta(z_i^k - z_i^\star)}_2^2} - 2n\eta^2\sigma^2
		\\
		&=
		\E{\norm{\mX^{k+1} - \mX^\star - \eta(\mZ^k - \mZ^\star)}_{\mI}^2} - 2n\eta^2\sigma^2\\
		&=
		\eta^2 \norm{\mZ^k - \mZ^\star}_\mI^2 + \E{\norm{\mX^{k+1} - \mX^\star}_\mI^2 
		-
		2\eta\dotprod{\mX^{k+1} - \mX^\star}{\mZ^k - \mZ^\star} }- 2n\eta^2\sigma^2.
		\end{align*}
		After rearranging, we get
		\begin{align*}
		-2\theta \E{\dotprod{\mX^{k+1} - \mX^\star}{\mZ^k - \mZ^\star}}
		&\leq
		-\eta \theta \norm{\mZ^k - \mZ^\star}_\mI^2
		-
		\frac{\theta\mu}{2}\E{\norm{\mX^{k+1} - \mX^\star}_\mI^2} + 2n\eta\theta\sigma^2\\
		&+
		(1-\eta\mu)\frac{\theta}{\eta}\norm{\mX^k - \mX^\star}_\mI^2
		-
		\left(1 - \frac{\eta\mu}{2}\right)\frac{\theta}{\eta}\E{\norm{\mX^{k+1} - \mX^\star}_\mI^2},
		\end{align*}
		which concludes the proof.
		
	\end{proof}

	\begin{theorem}[Convergence of Algorithm~\ref{alg:1}, Option C]\label{thm:C}
		Let $\Psi_C^k$ be a Lyapunov function defined as follows:
		\begin{equation}
		\Psi_C^k = \norm{\mZ^k - \mZ^\star}_{\mW^\dagger}^2 + \frac{(1 - \eta\mu/2)\theta}{\eta}\norm{\mX^k - \mX^\star}_\mI^2 + \frac{8\theta^2\omega \max_{(i,j)\in E}w_{ij}}{\alpha}
		\norm{\mH^{k} - \mX^\star}_\mI^2.
		\end{equation}
		Let $\rho_C$ be defined as follows:
		\begin{equation}\label{eq:29}
		\rho_C = \max \left\{2(\omega + 1), \frac{2(\lambda_{\max}(\mW) + 12\omega\max_{(i,j)\in E}w_{ij})}{\eta\mu \lambda_{\min}^+(\mW)}\right\}^{-1}.
		\end{equation}
		Choosing stepsize $\theta$
		\begin{equation}\label{eq:25}
		\theta =\frac{\mu}{2\lambda_{\max}(\mW) + 24\omega\max_{(i,j)\in E}w_{ij}},
		\end{equation}
		stepsize $\eta \leq \frac{1}{2L}$ and stepsize $\alpha = \frac{1}{\omega + 1}$ gives the following inequality:
		\begin{equation}\label{eq:28}
		\E{\Psi_C^{k+1}} \leq (1-\rho_C)\Psi_C^k +  2n\eta\theta\sigma^2.
		\end{equation}
		
	\end{theorem}
	
	\begin{proof}
		We start with rewriting \eqref{eq:2}:
		\begin{eqnarray*}
		\E{\norm{\mZ^{k+1} - \mZ^\star}_{\mW^{\dagger}}^2}
		&\overset{\eqref{eq:2}}{\leq}&
		\norm{\mZ^k - \mZ^\star}_{\mW^{\dagger}}^2
		+
		\E{-
			2\theta\dotprod{\mX^{k+1} - \mX^\star}{\mZ^k - \mZ^\star}
			+
			\theta^2
			\norm{\mX^{k+1} - \mX^\star}_\mW^2
			+
			\Sigma^k
		}\\
		&\overset{\eqref{eq:24}}{\leq}&
		\norm{\mZ^k - \mZ^\star}_{\mW^{\dagger}}^2
		-
		\eta\theta\norm{\mZ^k - \mZ^\star}_\mI^2
		-
		\frac{\theta\mu}{2}\E{\norm{\mX^{k+1} - \mX^\star}_\mI^2}
		+
		\theta^2\E{\norm{\mX^{k+1} - \mX^\star}_\mW^2}
		\\
		&+&
		\E{\Sigma^k}
		+
		(1-\eta\mu)\frac{\theta}{\eta}\norm{\mX^k - \mX^\star}_\mI^2
		-
		\left(1 - \frac{\eta\mu}{2}\right)\frac{\theta}{\eta}\E{\norm{\mX^{k+1} - \mX^\star}_\mI^2}
		+
		2n\eta\theta\sigma^2
		\\
		&\overset{\eqref{eq:5},\eqref{eq:6}}{\leq}&
		(
		1 - \eta\theta\lambda_{\min}^+(\mW)
		)
		\norm{\mZ^k - \mZ^\star}_{\mW^{\dagger}}^2
		-
		\theta\left(\frac{\mu}{2} - \theta\lambda_{\max}(\mW)\right)\E{\norm{\mX^{k+1} - \mX^\star}_\mI^2}
		+
		\E{\Sigma^k}\\
		&+&
		(1-\eta\mu)\frac{\theta}{\eta}\norm{\mX^k - \mX^\star}_\mI^2
		-
		\left(1 - \frac{\eta\mu}{2}\right)\frac{\theta}{\eta}\E{\norm{\mX^{k+1} - \mX^\star}_\mI^2}
		+
		2n\eta\theta\sigma^2
		\\
		&\overset{\eqref{eq:3}}{\leq}&
		(
		1 - \eta\theta\lambda_{\min}^+(\mW)
		)
		\norm{\mZ^k - \mZ^\star}_{\mW^{\dagger}}^2
		+
		4\theta^2\omega \max_{(i,j)\in E}w_{ij}\norm{\mH^{k} - \mX_\star}_{\mI}^2\\
		&-&
		\theta\left(
		\frac{\mu}{2} - \theta\left[\lambda_{\max}(\mW) + 4\omega\max_{(i,j)\in E}w_{ij} \right]
		\right)\E{\norm{\mX^{k+1} - \mX^\star}_\mI^2}\\
		&+&
		(1-\eta\mu)\frac{\theta}{\eta}\norm{\mX^k - \mX^\star}_\mI^2
		-
		\left(1 - \frac{\eta\mu}{2}\right)\frac{\theta}{\eta}\E{\norm{\mX^{k+1} - \mX^\star}_\mI^2}
		+
		2n\eta\theta\sigma^2
		\end{eqnarray*}
		Now, we combine this with \eqref{eq:7}:
		\begin{align*}
		\E{\Psi_C^{k+1}}
		&\leq
		(1 - \eta\theta\lambda_{\min}^+(\mW))\norm{\mZ^k - \mZ^\star}_{\mW^{\dagger}}^2
		+
		(1-\eta\mu)\frac{\theta}{\eta}\norm{\mX^k - \mX^\star}_\mI^2
		+
		4\theta^2\omega \max_{(i,j)\in E}w_{ij}\norm{\mH^{k} - \mX_\star}_{\mI}^2\\
		&-
		\theta\left(
		\frac{\mu}{2} - \theta\left[\lambda_{\max}(\mW) + 4\omega\max_{(i,j)\in E}w_{ij} \right]
		\right)\E{\norm{\mX^{k+1} - \mX^\star}_\mI^2}\\
		&+
		(1-\alpha)\frac{8\theta^2\omega \max_{(i,j)\in E}w_{ij}}{\alpha}
		\norm{\mH^{k} - \mX^\star}_\mI^2
		+
		8\theta^2\omega \max_{(i,j)\in E}w_{ij}
		\E{\norm{\mX^{k+1} - \mX^\star}_\mI^2}
		+
		2n\eta\theta\sigma^2\\
		&=
		(1 - \eta\theta\lambda_{\min}^+(\mW))\norm{\mZ^k - \mZ^\star}_{\mW^{\dagger}}^2
		+
		(1-\eta\mu)\frac{\theta}{\eta}\norm{\mX^k - \mX^\star}_\mI^2
		+
		2n\eta\theta\sigma^2\\
		&+
		\left(1-\frac{\alpha}{2}\right)\frac{8\theta^2\omega \max_{(i,j)\in E}w_{ij}}{\alpha}
		\norm{\mH^{k} - \mX^\star}_\mI^2\\
		&-
		\theta\left(
		\frac{\mu}{2} - \theta\left[\lambda_{\max}(\mW) + 12\omega\max_{(i,j)\in E}w_{ij} \right]
		\right)\E{\norm{\mX^{k+1} - \mX^\star}_\mI^2}.
		\end{align*}
		Using \eqref{eq:25} we get
		\begin{align*}
		\E{\Psi_C^{k+1}}
		&\leq
		(1 - \eta\theta\lambda_{\min}^+(\mW))\norm{\mZ^k - \mZ^\star}_{\mW^{\dagger}}^2
		+
		\left(1 - \frac{\eta\mu}{2 - \eta\mu}\right)\frac{(1 - \eta\mu/2)\theta}{\eta}\norm{\mX^k - \mX^\star}_\mI^2\\
		&+
		\left(1-\frac{\alpha}{2}\right)\frac{8\theta^2\omega \max_{(i,j)\in E}w_{ij}}{\alpha}
		\norm{\mH^{k} - \mX^\star}_\mI^2
		+
		2n\eta\theta\sigma^2\\
		&\leq
		(1 - \eta\theta\lambda_{\min}^+(\mW))\norm{\mZ^k - \mZ^\star}_{\mW^{\dagger}}^2
		+
		\left(1 - \frac{\eta\mu}{2}\right)\frac{(1 - \eta\mu/2)\theta}{\eta}\norm{\mX^k - \mX^\star}_\mI^2\\
		&+
		\left(1-\frac{\alpha}{2}\right)\frac{8\theta^2\omega \max_{(i,j)\in E}w_{ij}}{\alpha}
		\norm{\mH^{k} - \mX^\star}_\mI^2
		+
		2n\eta\theta\sigma^2\\
		&\leq
		(1-\rho_C)\Psi_C^k + 2n\eta\theta\sigma^2,
		\end{align*}
		which concludes the proof.
	\end{proof}

	\begin{corollary}\label{cor:C}
		Let
		\begin{equation}
			\hat{x}^k = \frac{1}{n}\sum_{i=1}^n x_i^k.
		\end{equation}
		Then for any $\epsilon > 0$, choosing stepsize
		\begin{equation}\label{eq:26}
			\eta = \min\left\{\frac{1}{2L},
			\frac{\sqrt{\epsilon}}{4\sigma\sqrt{\omega+1}},
			\frac{\epsilon\mu\lambda_{\min}^+(\mW)}{16\sigma^2(\lambda_{\max}(\mW)+ 12\omega\max_{(i,j)\in E}w_{ij})}\right\}
		\end{equation}
		and number of iterations
		\begin{align}
			k \geq \max&\left\{ 
				2(\omega+1),
				\frac{4L(\lambda_{\max}(\mW) + 12\omega\max_{(i,j)\in E}w_{ij})}{\mu\lambda_{\min}^+(\mW)},\right.\nonumber \\
				&\frac{8\sigma\sqrt{\omega+1}(\lambda_{\max}(\mW) + 12\omega\max_{(i,j)\in E}w_{ij})}{\mu\lambda_{\min}^+(\mW)\sqrt{\epsilon}}, \nonumber \\
				&\left.
				\frac{32\sigma^2(\lambda_{\max}(\mW) + 12\omega\max_{(i,j)\in E}w_{ij})^2}{\epsilon\mu^2(\lambda_{\min}^+(\mW))^2}
			\right\}\log\frac{4\eta\Psi_C^0}{n\theta\epsilon}\label{eq:27}
		\end{align}
		gives
		\begin{equation}
			\E{\norm{\hat{x}^k - x^\star}_2^2} \leq \epsilon.
		\end{equation}
	\end{corollary}

	\begin{proof}
		Using definition of $\hat{x}^k$ and \eqref{eq:28} we get
		\begin{align*}
			\E{\norm{\hat{x}^k - x^\star}_2^2}
			&\leq
			\frac{1}{n}\E{\norm{\mX^k - \mX^\star}_{\mI}^2}
			\leq
			\E{\frac{\eta}{n\theta(1-\eta\mu/2)}\Psi_C^k}
			\leq
			\frac{2\eta}{n\theta}\E{\Psi_C^k}\\
			&\leq
			\frac{2\eta}{n\theta}\left((1-\rho_C)^k\Psi_C^0 + \frac{2n\eta\theta\sigma^2}{\rho_C}\right)
			=
			(1-\rho_C)^k\frac{2\eta}{n\theta}\Psi_C^0 + \frac{4\eta^2\sigma^2}{\rho_C}.
		\end{align*}
		From \eqref{eq:26} and \eqref{eq:29} it follows that $\frac{4\eta^2\sigma^2}{\rho_C} \leq \frac{\epsilon}{2}$ and hence
		\begin{align*}
			\E{\norm{\hat{x}^k - x^\star}_2^2}
			\leq
			(1-\rho_C)^k\frac{2\eta}{n\theta}\Psi_C^0 + \frac{\epsilon}{2}.
		\end{align*}
		From \eqref{eq:27}, \eqref{eq:26}  and \eqref{eq:29} it follows that $(1-\rho_C)^k \leq \frac{n\theta\epsilon}{4\eta \Psi_C^0}$ and hence
		\begin{align*}
			\E{\norm{\hat{x}^k - x^\star}_2^2}
			\leq
			\frac{\epsilon}{2}+ \frac{\epsilon}{2} = \epsilon.
		\end{align*}
	\end{proof}

	\subsection{Option D}
	
	\begin{lemma}[Primal step, Option D]
		Let $\eta \leq \frac{1}{6L}$. Then the following inequality holds:
		\begin{align}
			-2\theta \E{\dotprod{\mX^{k+1} - \mX^\star}{\mZ^k - \mZ^\star}}
			&\leq
			-\eta \theta \norm{\mZ^k - \mZ^\star}_\mI^2
			-
			\frac{\theta\mu}{2}\E{\norm{\mX^{k+1} - \mX^\star}_\mI^2}\nonumber\\
			&+
			(1-\eta\mu)\frac{\theta}{\eta}\norm{\mX^k - \mX^\star}_\mI^2
			-
			\left(1 - \frac{\eta\mu}{2}\right)\frac{\theta}{\eta}\E{\norm{\mX^{k+1} - \mX^\star}_\mI^2}\nonumber\\
			&+
			8mL\eta\theta\sum_{i=1}^n\left[\left(1 - \frac{1}{2m}\right)B_{f_i}(w_i^k,x^\star) - \E{B_{f_i}(w_i^{k+1},x^\star)} \right].\label{eq:30}
		\end{align}
	\end{lemma}
	
	\begin{proof}
		From Line~\ref{line:4} of Algorithm~\ref{alg:1} it follows that
		\begin{eqnarray*}
			\E{\norm{x_i^{k+1} - x^\star - \eta(z_i^k - z_i^\star)}_2^2}
			&=&
			\E{\norm{x_i^k - x^\star - \eta (\nabla f_{ij_i^k}(x_i^k) - \nabla f_{ij_i^k}(w_i^k) + \nabla f_i(w_i^k) - \nabla f_i(x^\star)) }_2^2}\\
			&=&
			\norm{x_i^k - x^\star}_2^2 
			-
			2\eta\dotprod{\nabla f_i(x_i^k) - \nabla f_i(x^\star)}{x_i^k - x^\star}\\
			&+&
			\eta^2\E{\norm{\nabla f_{ij_i^k}(x_i^k) - \nabla f_{ij_i^k}(w_i^k) + \nabla f_i(w_i^k) - \nabla f_i(x^\star)}_2^2}\\
			&\overset{\eqref{eq:22}}{\leq}&
			(1-\eta\mu)\norm{x_i^k - x^\star}_2^2
			-
			2\eta B_{f_i}(x_i^k, x^\star)
			+
			2\eta^2\E{\norm{\nabla f_{ij_i^k}(x_i^k) - \nabla f_{ij_i^k}(x^\star) }_2^2}\\
			&+&
			2\eta^2\E{\norm{\nabla f_{ij_i^k}(x^\star)  - \nabla f_i(x^\star) - \nabla f_{ij_i^k}(w_i^k) + \nabla f_i(w_i^k)}_2^2}\\
			&\leq&
			(1-\eta\mu)\norm{x_i^k - x^\star}_2^2
			-
			2\eta B_{f_i}(x_i^k, x^\star)
			+
			\frac{2\eta^2}{m}\sum_{j=1}^m \norm{\nabla f_{ij}(x_i^k) - \nabla f_{ij}(x^\star) }_2^2\\
			&+&
			\frac{2\eta^2}{m}\sum_{j=1}^m \norm{\nabla f_{ij}(w_i^k) - \nabla f_{ij}(x^\star) }_2^2\\
			&\overset{\eqref{eq:31}}{\leq}&
			(1-\eta\mu)\norm{x_i^k - x^\star}_2^2
			-
			2\eta B_{f_i}(x_i^k, x^\star)
			+
			4L\eta^2 B_{f_i}(x_i^k, x^\star)
			+
			4L\eta^2 B_{f_i}(w_i^k, x^\star).
		\end{eqnarray*}
		From Line~\ref{line:5} of Algorithm~\ref{alg:1} it follows that
		\begin{align*}
			\E{B_{f_i}(w_i^{k+1},x^\star)} =\left(1 - \frac{1}{m}\right)B_{f_i}(w_i^k,x^\star) +  \frac{1}{m}B_{f_i}(x_i^k, x_i^\star),
		\end{align*}
		which gives
		\begin{align*}
			B_{f_i}(w_i^k, x_i^\star) \leq 2m\left[\left(1 - \frac{1}{2m}\right)B_{f_i}(w_i^k,x^\star) - \E{B_{f_i}(w_i^{k+1},x^\star)} \right]+ 2B_{f_i}(x_i^k, x_i^\star),
		\end{align*}
		and hence using stepsize $\eta \leq \frac{1}{6L}$ we get
		\begin{align*}
			\E{\norm{x_i^{k+1} - x^\star - \eta(z_i^k - z_i^\star)}_2^2}
			&\leq
			(1-\eta\mu)\norm{x_i^k - x^\star}_2^2
			-
			2\eta B_{f_i}(x_i^k, x^\star)
			+
			12L\eta^2 B_{f_i}(x_i^k, x^\star)\\
			&+
			8mL\eta^2 \left[\left(1 - \frac{1}{2m}\right)B_{f_i}(w_i^k,x^\star) - \E{B_{f_i}(w_i^{k+1},x^\star)} \right]\\
			&\leq			
			(1-\eta\mu)\norm{x_i^k - x^\star}_2^2
			+
			8mL\eta^2 \left[\left(1 - \frac{1}{2m}\right)B_{f_i}(w_i^k,x^\star) - \E{B_{f_i}(w_i^{k+1},x^\star)} \right].
		\end{align*}
		This leads to
		\begin{align*}
			(1-\eta\mu)\norm{\mX^k - \mX^\star}_\mI^2
			&= 
			\sum_{i=1}^n  (1-\eta\mu)\norm{x_i^k - x_i^\star}_2^2
			\geq
			\sum_{i=1}^n  \E{\norm{x_i^{k+1} - x^\star - \eta(z_i^k - z_i^\star)}_2^2}\\
			&-
			8mL\eta^2 \sum_{i=1}^n\left[\left(1 - \frac{1}{2m}\right)B_{f_i}(w_i^k,x^\star) - \E{B_{f_i}(w_i^{k+1},x^\star)} \right]
			\\
			&=
			\E{\norm{\mX^{k+1} - \mX^\star - \eta(\mZ^k - \mZ^\star)}_{\mI}^2}\\
			&-
			8mL\eta^2 \sum_{i=1}^n\left[\left(1 - \frac{1}{2m}\right)B_{f_i}(w_i^k,x^\star) - \E{B_{f_i}(w_i^{k+1},x^\star)} \right]\\
			&=
			\E{\norm{\mX^{k+1} - \mX^\star}_\mI^2} + \eta^2 \norm{\mZ^k - \mZ^\star}_\mI^2
			-
			2\eta\E{\dotprod{\mX^{k+1} - \mX^\star}{\mZ^k - \mZ^\star}}\\
			&-
			8mL\eta^2 \sum_{i=1}^n\left[\left(1 - \frac{1}{2m}\right)B_{f_i}(w_i^k,x^\star) - \E{B_{f_i}(w_i^{k+1},x^\star)} \right].
		\end{align*}
		After rearranging, we get
		\begin{align*}
			-2\theta \E{\dotprod{\mX^{k+1} - \mX^\star}{\mZ^k - \mZ^\star}}
			&\leq
			-\eta \theta \norm{\mZ^k - \mZ^\star}_\mI^2
			-
			\frac{\theta\mu}{2}\E{\norm{\mX^{k+1} - \mX^\star}_\mI^2}\\
			&+
			(1-\eta\mu)\frac{\theta}{\eta}\norm{\mX^k - \mX^\star}_\mI^2
			-
			\left(1 - \frac{\eta\mu}{2}\right)\frac{\theta}{\eta}\E{\norm{\mX^{k+1} - \mX^\star}_\mI^2}\\
			&+
			8mL\eta\theta\sum_{i=1}^n\left[\left(1 - \frac{1}{2m}\right)B_{f_i}(w_i^k,x^\star) - \E{B_{f_i}(w_i^{k+1},x^\star)} \right],
		\end{align*}
		which concludes the proof.
		
	\end{proof}
	
	\begin{theorem}[Convergence of Algorithm~\ref{alg:1}, Option D]\label{thm:D}
		Let $\Psi_D^k$ be a Lyapunov function which is defined as follows:
		\begin{equation}
			\Psi_D^k =\norm{\mZ^k - \mZ^\star}_{\mW^\dagger}^2 + \frac{(1 - \eta\mu/2)\theta}{\eta}\norm{\mX^k - \mX^\star}_\mI^2 + \frac{8\theta^2\omega \max_{(i,j)\in E}w_{ij}}{\alpha}
			\norm{\mH^{k} - \mX^\star}_\mI^2
			+
			8mL\eta\theta\sum_{i=1}^nB_{f_i}(w_i^k,x^\star).
		\end{equation}
		Let $\rho_D$ be defined as follows:
		\begin{equation}
			\rho_D = \max \left\{2m, 2(\omega + 1), \frac{12L(\lambda_{\max}(\mW) + 12\omega\max_{(i,j)\in E}w_{ij})}{\mu \lambda_{\min}^+(\mW)}\right\}^{-1}.
		\end{equation}
		Choosing stepsize $\theta$
		\begin{equation}\label{eq:32}
			\theta = \frac{\mu}{2\lambda_{\max}(\mW) + 24\omega\max_{(i,j)\in E}w_{ij}},
		\end{equation}
		stepsize $\eta = \frac{1}{6L}$ and stepsize $\alpha = \frac{1}{\omega + 1}$ gives the following inequality:
		\begin{equation}
			\E{\Psi_D^{k+1}} \leq (1-\rho_D)\Psi_D^k.
		\end{equation}
		
	\end{theorem}
	
	\begin{proof}
		We start with rewriting \eqref{eq:2}:
		\begin{eqnarray*}
			\E{\norm{\mZ^{k+1} - \mZ^\star}_{\mW^{\dagger}}^2}
			&\overset{\eqref{eq:2}}{\leq}&
			\norm{\mZ^k - \mZ^\star}_{\mW^{\dagger}}^2
			+
			\E{-
				2\theta\dotprod{\mX^{k+1} - \mX^\star}{\mZ^k - \mZ^\star}
				+
				\theta^2
				\norm{\mX^{k+1} - \mX^\star}_\mW^2
				+
				\Sigma^k
			}\\
			&\overset{\eqref{eq:30}}{\leq}&
			\norm{\mZ^k - \mZ^\star}_{\mW^{\dagger}}^2
			-
			\eta\theta\norm{\mZ^k - \mZ^\star}_\mI^2
			-
			\E{\frac{\theta\mu}{2}\norm{\mX^{k+1} - \mX^\star}_\mI^2
			+
			\theta^2\norm{\mX^{k+1} - \mX^\star}_\mW^2
			+
			\Sigma^k}\\
			&+&
			(1-\eta\mu)\frac{\theta}{\eta}\norm{\mX^k - \mX^\star}_\mI^2
			-
			\left(1 - \frac{\eta\mu}{2}\right)\frac{\theta}{\eta}\E{\norm{\mX^{k+1} - \mX^\star}_\mI^2}
			\\
			&+&
			8mL\eta\theta\sum_{i=1}^n\left[\left(1 - \frac{1}{2m}\right)B_{f_i}(w_i^k,x^\star) - \E{B_{f_i}(w_i^{k+1},x^\star)} \right]\\
			&\overset{\eqref{eq:5},\eqref{eq:6}}{\leq}&
			(
			1 - \eta\theta\lambda_{\min}^+(\mW)
			)
			\norm{\mZ^k - \mZ^\star}_{\mW^{\dagger}}^2
			-
			\theta\left(\frac{\mu}{2} - \theta\lambda_{\max}(\mW)\right)\E{\norm{\mX^{k+1} - \mX^\star}_\mI^2}
			+
			\E{\Sigma^k}\\
			&+&
			(1-\eta\mu)\frac{\theta}{\eta}\norm{\mX^k - \mX^\star}_\mI^2
			-
			\left(1 - \frac{\eta\mu}{2}\right)\frac{\theta}{\eta}\E{\norm{\mX^{k+1} - \mX^\star}_\mI^2}
			\\
			&+&
			8mL\eta\theta\sum_{i=1}^n\left[\left(1 - \frac{1}{2m}\right)B_{f_i}(w_i^k,x^\star) - \E{B_{f_i}(w_i^{k+1},x^\star)} \right]\\
			&\overset{\eqref{eq:3}}{\leq}&
			(
			1 - \eta\theta\lambda_{\min}^+(\mW)
			)
			\norm{\mZ^k - \mZ^\star}_{\mW^{\dagger}}^2
			+
			4\theta^2\omega \max_{(i,j)\in E}w_{ij}\norm{\mH^{k} - \mX_\star}_{\mI}^2\\
			&-&
			\theta\left(
			\frac{\mu}{2} - \theta\left[\lambda_{\max}(\mW) + 4\omega\max_{(i,j)\in E}w_{ij} \right]
			\right)\E{\norm{\mX^{k+1} - \mX^\star}_\mI^2}\\
			&+&
			(1-\eta\mu)\frac{\theta}{\eta}\norm{\mX^k - \mX^\star}_\mI^2
			-
			\left(1 - \frac{\eta\mu}{2}\right)\frac{\theta}{\eta}\E{\norm{\mX^{k+1} - \mX^\star}_\mI^2}\\
			&+&
			8mL\eta\theta\sum_{i=1}^n\left[\left(1 - \frac{1}{2m}\right)B_{f_i}(w_i^k,x^\star) - \E{B_{f_i}(w_i^{k+1},x^\star)} \right].
		\end{eqnarray*}
		Now, we combine this with \eqref{eq:7}:
		\begin{align*}
			\E{\Psi_D^{k+1}}
			&\leq
			(1 - \eta\theta\lambda_{\min}^+(\mW))\norm{\mZ^k - \mZ^\star}_{\mW^{\dagger}}^2
			+
			(1-\eta\mu)\frac{\theta}{\eta}\norm{\mX^k - \mX^\star}_\mI^2
			+
			\left(1 - \frac{1}{2m}\right)8mL\eta\theta\sum_{i=1}^nB_{f_i}(w_i^k,x^\star)
			\\
			&+
			4\theta^2\omega \max_{(i,j)\in E}w_{ij}\norm{\mH^{k} - \mX_\star}_{\mI}^2
			-
			\theta\left(
			\frac{\mu}{2} - \theta\left[\lambda_{\max}(\mW) + 4\omega\max_{(i,j)\in E}w_{ij} \right]
			\right)\norm{\mX^{k+1} - \mX^\star}_\mI^2\\
			&+
			(1-\alpha)\frac{8\theta^2\omega \max_{(i,j)\in E}w_{ij}}{\alpha}
			\norm{\mH^{k} - \mX^\star}_\mI^2
			+
			8\theta^2\omega \max_{(i,j)\in E}w_{ij}
			\norm{\mX^{k+1} - \mX^\star}_\mI^2\\
			&=
			(1 - \eta\theta\lambda_{\min}^+(\mW))\norm{\mZ^k - \mZ^\star}_{\mW^{\dagger}}^2
			+
			(1-\eta\mu)\frac{\theta}{\eta}\norm{\mX^k - \mX^\star}_\mI^2\\
			&+
			\left(1 - \frac{1}{2m}\right)8mL\eta\theta\sum_{i=1}^nB_{f_i}(w_i^k,x^\star)
			+
			\left(1-\frac{\alpha}{2}\right)\frac{8\theta^2\omega \max_{(i,j)\in E}w_{ij}}{\alpha}
			\norm{\mH^{k} - \mX^\star}_\mI^2\\
			&-
			\theta\left(
			\frac{\mu}{2} - \theta\left[\lambda_{\max}(\mW) + 12\omega\max_{(i,j)\in E}w_{ij} \right]
			\right)\norm{\mX^{k+1} - \mX^\star}_\mI^2.
		\end{align*}
		Using \eqref{eq:32} we get
		\begin{align*}
			\E{\Psi_D^{k+1}}
			&\leq
			(1 - \eta\theta\lambda_{\min}^+(\mW))\norm{\mZ^k - \mZ^\star}_{\mW^{\dagger}}^2
			+
			\left(1 - \frac{\eta\mu}{2 - \eta\mu}\right)\frac{(1 - \eta\mu/2)\theta}{\eta}\norm{\mX^k - \mX^\star}_\mI^2\\
			&+
			\left(1-\frac{\alpha}{2}\right)\frac{8\theta^2\omega \max_{(i,j)\in E}w_{ij}}{\alpha}
			\norm{\mH^{k} - \mX^\star}_\mI^2
			+
			\left(1 - \frac{1}{2m}\right)8mL\eta\theta\sum_{i=1}^nB_{f_i}(w_i^k,x^\star)\\
			&\leq
			(1 - \eta\theta\lambda_{\min}^+(\mW))\norm{\mZ^k - \mZ^\star}_{\mW^{\dagger}}^2
			+
			\left(1 - \frac{\eta\mu}{2}\right)\frac{(1 - \eta\mu/2)\theta}{\eta}\norm{\mX^k - \mX^\star}_\mI^2\\
			&+
			\left(1-\frac{\alpha}{2}\right)\frac{8\theta^2\omega \max_{(i,j)\in E}w_{ij}}{\alpha}
			\norm{\mH^{k} - \mX^\star}_\mI^2
			+
			\left(1 - \frac{1}{2m}\right)8mL\eta\theta\sum_{i=1}^nB_{f_i}(w_i^k,x^\star)\\
			&\leq
			(1-\rho_D)\Psi_D^k,
		\end{align*}
		which concludes the proof.
	\end{proof}